\PassOptionsToPackage{prologue,dvipsnames}{xcolor}

\documentclass[12pt]{amsart}

\usepackage[latin1]{inputenc}
\usepackage[active]{srcltx}
\usepackage{fullpage}
\usepackage{amsfonts,enumerate}
\usepackage{comment, enumerate}
\usepackage[mathscr]{euscript}
\usepackage{epsfig}
\usepackage{latexsym}
\usepackage[all]{xy}
\usepackage{amssymb}
\usepackage{amsthm}
\usepackage{amsmath}
\usepackage{amsxtra}
\usepackage{float}
\usepackage{tikz}
\usepackage{tikz-cd}

\usepackage[dvipsnames]{xcolor}
\usepackage[tickmarkheight=.5em,textwidth=\marginparwidth,textsize=small]{todonotes}
\usepackage[backref=page]{hyperref}
\hypersetup{citecolor=MidnightBlue,
	colorlinks=true,
	linkcolor=MidnightBlue,
	urlcolor=MidnightBlue}

.tfm
.tfm

\theoremstyle{plain}
\newtheorem{prop}{Proposition}[section]
\newtheorem{introthm}{Theorem}
\newtheorem{thm}[prop]{Theorem}
\newtheorem{coro}[prop]{Corollary}

\newtheorem{lemma}[prop]{Lemma}
\newtheorem{conj}{Conjecture}

\newtheorem{thmA}{Theorem}

\newtheorem{thmB}{Theorem}

\newtheorem{thmC}{Theorem}

\theoremstyle{definition}
\newtheorem{defi}[prop]{Definition}

\theoremstyle{remark}

\newtheorem{exm}[prop]{Example}
\newtheorem{remark}[prop]{Remark}

\numberwithin{table}{section}

\DeclareMathOperator{\Jac}{Jac}

\DeclareMathOperator{\cond}{cond}
\DeclareMathOperator{\Frob}{Frob}

\DeclareMathOperator{\End}{End}

\DeclareMathOperator{\Gal}{Gal}

\DeclareMathOperator{\Ind}{Ind}

\DeclareMathOperator{\ord}{ord}

\newcommand{\bfP}{\mathcal P}

\newcommand{\F}{\mathbb F}
\newcommand{\bfC}{\mathcal C}
\newcommand{\bfD}{\mathcal D}

\newcommand{\J}{\mathcal J}

\newcommand{\Om}{{\mathcal{O}}}

\newcommand{\Disc}{\Delta}

\newcommand{\GL}{{\rm GL}}

\newcommand{\HGM}{\mathcal H}

\def\gent{\chi_{\idF}}

\newcommand{\newform}{\mathscr F}
\def\mot{\mathfrak h}
\newcommand{\mubb}{\mbox{$\raisebox{-0.59ex}
  {$l$}\hspace{-0.18em}\mu\hspace{-0.88em}\raisebox{-0.98ex}{\scalebox{2}
  {$\color{white}.$}}\hspace{-0.416em}\raisebox{+0.88ex}
  {$\color{white}.$}\hspace{0.46em}$}{}}

\makeatletter

\newcommand{\verbatimfont}[1]{%
	\renewcommand{\verbatim@font}{\ttfamily #1}%
}
\renewcommand{\verbatim@font}{\ttfamily\small}
\makeatother

\def\ZZ{\mathbb Z}

\def\QQ{\mathbb Q}

\def\PP{\mathbb P}

\def\CC{\mathbb C}

\def\<#1>{{\left\langle{#1}\right\rangle}}

\def\Z{{\mathbb Z}}             
\def\Q{{\mathbb Q}}             
\def\idK{{\mathfrak l}}               
\def\idF{{\lambda}}               

\def\id#1{{\mathfrak{#1}}}      
\def\normid#1{{\norm{(\id{#1})}}}
\DeclareMathOperator{\norm}{{\mathscr N}}

\DeclareMathOperator{\trace}{{\mathrm{Tr}}}

\newcommand{\lmfdbec}[3]{\href{http://www.lmfdb.org/EllipticCurve/Q/#1/#2/#3}{{\text{\rm#1-#2#3}}}}

\begin{document}

\title{On the generalized Fermat equation of signature $(5,p,3)$}




\author{Ariel Pacetti}
\address{Center for Research and Development in Mathematics and Applications (CIDMA),
	Department of Mathematics, University of Aveiro, 3810-193 Aveiro, Portugal}
\email{apacetti@ua.pt}
\thanks{The first author is supported by CIDMA under the Portuguese
  Foundation for Science and Technology (FCT,
  https://ror.org/00snfqn58) Multi-Annual Financing Program for R\&D
  Units, grants UID/4106/2025 and UID/PRR/4106/2025.}

\author{Lucas Villagra Torcomian}
\address{Department of Mathematics, Simon Fraser University, 8888 University Drive, Burnaby,
British Columbia, V5A 1S6, Canada.}
\email{lvillagr@sfu.ca}

\thanks{The second author is supported by a PIMS-Simons postdoctoral fellowship.}


\begin{abstract}
  In this article we study solutions to the generalized Fermat
  equation $x^q+y^p+z^r=0 $ using hypergeometric motives within the
  framework of the modular method. In doing so, we give an explicit
  description of the ramification behavior at primes dividing $2qr$
  and analyze the contribution of trivial solutions. We identify a
  general obstruction to the modular method that accounts for its
  failure in many instances. As an application, assuming a standard
  large image conjecture, we prove that the previous equation
  admits no nontrivial primitive solutions $(a,b,c)$ with $3 \nmid c$,
  when $q=5,$ $r=3$ and $p$ is a prime sufficiently large.
\end{abstract}


\maketitle
{
	\hypersetup{linkcolor=black}
	\setcounter{tocdepth}{1}
	\tableofcontents
}

\section{Introduction}

Fermat's Last Theorem inspired a broader investigation of equations of similar form, most notably the \emph{generalized Fermat equation}
\begin{equation}
	\label{eq:GFE}
	x^q + y^p + z^r = 0.
\end{equation}
A solution $(a,b,c)\in\ZZ^3$ to~\eqref{eq:GFE} is called
\emph{primitive} if $\gcd(a,b,c)=1$, and is said to be \emph{trivial}
if $abc=0$.  We refer to the triple of exponents $(q,p,r)$ as the
\emph{signature} of~\eqref{eq:GFE}.  It is expected (see Beal's
conjecture in \cite{Beal}) that if $\max\{q,p,r\}>2$, then there are
no non-trivial primitive solutions to (\ref{eq:GFE}). We recommend the
reader to look at the very nice survey \cite{MR3526934}.

The strategy followed by Wiles in \cite{Wiles}, based on Galois
representations and modularity, is commonly referred to as the
\emph{modular method}. Its main steps are the following.
\begin{enumerate}[(I)]
\item Attach to a putative solution to (\ref{eq:GFE}) a
  geometric object $\bfC$ defined over a small degree number field $K$
  so that the $p$-th member of the family of Galois representations
  attached to $\bfC$ has small residual ramification set
  (independent of the solution).
  
\item Prove that the geometric object is \emph{modular}, i.e.\ its
  $L$-series matches one coming from an automorphic form.

\item Prove that the $p$-th residual representation in Step (I) is
  absolutely irreducible, and compute its conductor $\id{n}$. 
  
\item Compute the space of newforms of level $\id{n}$ and prove that
  none of them is related to a solution to (\ref{eq:GFE}).
\end{enumerate}
The modular method is a remarkably powerful tool for proving
non-existence of solutions while varying along a line in the space
of signatures. For instance, the line $q=p=r$ corresponds to Fermat's
Last Theorem. Our choice of notation for the exponents
in~\eqref{eq:GFE} reflects the fact that we allow $p$ to be the
varying parameter, in accordance with the notation used in \cite{GP}.

In \cite{Darmon}, Darmon described a program based on the modular
method to prove that the generalized Fermat equation~\eqref{eq:GFE}
has no non-trivial solutions when one of the exponents is a varying
prime $p\ge 3$.  Darmon's idea consists on constructing a family of
residual Galois representations (indexed by a parameter $t$), called
\emph{Frey representations}, with the property that if $(a,b,c)$ is a
solution to~\eqref{eq:GFE}, then the specialization at $t_0=-a^q/c^r$
is unramified outside  a small and explicitly controlled set of primes,
namely those dividing $2qr$. The specialized representation coincides
with the one obtained in Step (I) of the original modular method
approach.
%
Darmon's program has been successfully applied to the cases of
signatures $(5,p,5)$ \cite{BCDF1} and $(p,p,5)$ \cite{ChenKoutsianas}.
Until now, these applications have focused on situations where the
construction of the Frey representation is explicit,  via
Frey abelian varieties arising as the Jacobians of Frey hyperelliptic
curves.

\subsection{Hypergeometric motives: a new approach}

In the recent article \cite{GP}, Golfieri and the first author show
how the geometric objects traditionally used as sources of Galois
representations in Step (I) can be replaced by a more combinatorial one, namely a
hypergeometric motive.
Within this new context, a first goal of the present article is to
give an explicit formula for the conductor at the so-called
\emph{potentially wild primes}, as required in Step~(III).

{\color{MidnightBlue}
\begin{introthm}
	\label{thm:conductor-general}
	\color{black}
  Attached to a hypothetical solution $(a,b,c)$ to
  \[
    x^q+y^p+z^r, \ \text{ with } q>r>2, \ p>2,
  \]
  there exist two Galois representations 
  \[
    {\rho}^\pm_\id{p} : \Gal(\bar{K}/K) \to \GL_2(K_\id{p}),
  \]
  where $K=\Q(\zeta_q)^+\cdot\Q(\zeta_r)^+$ and $\id{p}$ is a prime in $K$ above $p$, satisfying the following. If $\bar{\rho}^\pm_\id{p}$ is irreducible, then:
  \begin{enumerate}[(i)]
  \item $\bar{\rho}^+_{\id{p}}$ is unramified outside $pqr$ and is finite at $p$
  \item $\bar{\rho}^-_{\id{p}}$ is unramified outside $2pqr$ and is finite at $p$.
  \end{enumerate}
  The conductor exponent at primes dividing $qr$ is described in
  Corollary \ref{coro:conductor-at-r}, Proposition
  \ref{prop:conductor-at-q}, and the conductor exponent at $2$ for
  $\bar{\rho}^-_{\mathfrak{p}}$ is given in
  Proposition~\ref{prop:cond-at-2}.
\end{introthm}

}

\subsection{A general obstruction}
The existence of the trivial solutions
\[
\{(\pm 1, \mp 1,0),(\pm 1,0,\mp 1), (0, \pm 1,\mp 1)\}
\]
to equation~\eqref{eq:GFE} constitutes an obstruction to the modular
method. Following the approach of \cite{GP}, we prove that only the first
and last pair of solutions correspond to abelian varieties with
complex multiplication, and we provide explicit formulas for their
conductors (see Lemmas~\ref{lemma:cond-at-t=0}
and~\ref{lemma:cond-at-t=inf}).

It is well known that sometimes the modular method fails to prove
non-existence of non-trivial primitive solutions.  In some cases (like
for signature $(3,p,2)$) the failure can be explained by the existence
of particular solutions to (\ref{eq:5p3}), but in others the reason
for the failure remains mysterious.

If $(a,c,s,\ell,m,n,u,v) \in \Z^2 \times \Z_{\ge 0}^6$, with
$ac \neq 0$, satisfies the relation
\begin{equation}
  \label{eq:ghost-eqn}
q^s r^\ell a^q \pm q^m r^n+ q^ur^vc^r=0,
\end{equation}
then the specialization of Frey's representation (and also of the
associated hypergeometric motive) at $t_0:=-a^qq^{s-u}r^{\ell-v}/c^r$
is also unramified outside $qr$. We call solutions to
(\ref{eq:ghost-eqn}) \emph{ghost solutions}; they are studied in
\S\ref{section:obstruction}.  If the motive attached to a ghost
solution is modular and its conductor lies among the possible
conductors of Hilbert modular forms that must be eliminated in the
modular method, then the method fails.  This phenomenon for Frey
representations appears not to have been previously observed.

In the particular case $q=5$, $r=3$ we found several solutions
to~\eqref{eq:ghost-eqn}, listed in Table~ \ref{table:ghost-sols}.  All of them are modular, and two of them lie
in the space of Hilbert modular forms
$S_2(\Gamma_0(3^3 \cdot (\sqrt{5})^3))$ over $\Q(\sqrt{5})$, one of
the spaces that must be analyzed in the elimination step. In the
present article we prove that all other forms without complex
multiplication can be discarded using classical elimination arguments.
It is unclear to us whether ghost solutions explain in general the
failure of the modular method, namely that any non-CM form that cannot be
discarded arises from a specialization coming from a solution
to~\eqref{eq:ghost-eqn}.

\subsection{Signature $(5,p,3)$} The main focus of the present article is the study of solutions to the generalized Fermat equation
\begin{equation}\label{eq:5p3}
	x^5 + y^p + z^3 = 0.
\end{equation}
According to Beal's conjecture, equation~(\ref{eq:5p3}) is expected to admit only trivial primitive solutions when $p \ge 3$. For $p = 2$, it is known that there are infinitely many solutions, and they are parametrized in \cite{Edwards}. The cases $p = 3,4,5$ were completely solved in \cite{Bruin,SiksekStoll,Poonen}, respectively. Therefore, we may assume that $p \ge 7$.

Equation (\ref{eq:5p3}) is a case  of particular interest for several reasons.  First, modularity of the hypergeometric motive
attached to any solution has been established in
\cite[Theorem~6.3]{GP}.  Secondly, the associated spaces of Hilbert
modular forms are defined over $K=\mathbb{Q}(\sqrt{5})$, making the
required computations feasible in practice.

Following the approach of \cite{GP}, we show that there is a
hyperelliptic realization of the motive. That is, there exists a
hyperelliptic curve $\bfC^{+}_{5,3}(t)$ defined over $\mathbb{Q}(t)$
that gives rise to a Frey representation, when considered over $K$.
The representation evaluated at $t_0 = -a^{5}/c^{3}$
coincides with the representation ${\rho}^{+}_{\mathfrak{p}}$
appearing in Theorem~\ref{thm:conductor-general}.  The following
result summarizes the results and the obstructions obtained while
studying solutions of~\eqref{eq:GFE} for the signature $(5,p,3)$.

{\color{MidnightBlue}
\begin{thmA}\label{thmA}
  \color{black} Set
  $\bfP:=\{2,3,5,7,11,13,19,29,31,41,61,71,79,89,101,109\}$. Let $p$ be a prime
  number not in $\bfP$ and let $(a,b,c)$ be a primitive solution to
  \begin{equation*}
    x^5 + y^p + z^3 =0
  \end{equation*}
  Set $t_0:=-a^5/c^3$. Then one of the following holds:
  \begin{enumerate}
  \item The residual Galois representation
    $\overline{\rho}_{\id{p}}^+$ is reducible.
  \item The Galois representation $\rho_{\id{p}}^+$ is congruent
    modulo $p$ to the Galois representation associated to a Hilbert modular form of $\Q(\sqrt{5})$ with
    complex multiplication by $\Q(\sqrt{-3})$ or by
    $\Q(\zeta_{15})/\Q(\sqrt{5})$.
  \item The Galois representation $\rho_{\id{p}}^+$ is congruent
    modulo $p$ to ${\rho}_{\bfC_{5,3}^+(t_0),\id{p}}$, where $t_0\in\{-\frac{1}{8}, \frac{9}{8}\}$ (these two values correspond to ghost solutions).
	\end{enumerate}
\end{thmA}
} The set $\bfP$ in Theorem \ref{thmA} arises from a standard elimination step
(based on ideas due to Mazur, as explained in \S\ref{section:Mazur},
using prime ideals of $\Q(\sqrt{5})$ of norm up to $400$). With the current standard tools of the modular method, we do not expect that this set can be reduced further (see Remark \ref{remar:Kraus}).

In theory the first case of Theorem \ref{thmA} can be discarded when $p$ is
large enough (as it is expected that only small primes can produce
reducible representations). Assuming some extra hypotheses on the
solution we can produce an explicit bound $C$
such that the residual representation $\bar{\rho}_{\id{p}}^+$ is
irreducible for all primes $p > C$ (see
Corollary~\ref{coro:large-image-5p3}). This allows to prove the
following corollary of Theorem \ref{thmA}.

{\color{MidnightBlue}
\begin{thmB}\label{thmB}
  \color{black} Let $\bfP:=\{2, 7, 11, 13, 19, 29, 31, 41, 61\}$. Then if $p$ is a prime
  number not in $\bfP$, there are no primitive solutions $(a,b,c)$ to
  \[
 x^5 + y^p + z^3 =0
 \]
 with $b$ odd and $3\mid b$ or $5 \mid b$.
\end{thmB}
} The hypothesis $b$ being odd is used to prove that the residual
representation $\overline{\rho}_{\id{p}}^+$ is absolutely irreducible,
while the hypothesis $3 \mid b$ or $5 \mid b$ is used to rule out the
second case in Theorem~A, since it implies that the Hilbert modular
form attached to a solution has special local type at $3$ or at
$(\sqrt{5})$. Consequently, on the one hand the form does not have
complex multiplication, and on the other hand it does not lie in the
space of the two newforms appearing in the third case of Theorem
A.

\vspace{5pt}

It is expected that Frey varieties without complex multiplication have
residual Galois image as large as possible. More concretely, the
following is expected to hold (see \cite[Conjecture 4.1]{Darmon}).

\begin{conj}\label{conj:large-image}
  Let $K$ be a totally real field and $L$ a number field. There exists
  a constant $C(L,K)$, depending only on $L$ and $K$, such that for
  any abelian variety $A/L$ of $\mathrm{GL}_2$-type with
	\[
	\End_L(A)\otimes \Q = \End_{\overline{L}}(A)\otimes \Q \simeq K,
	\]
	and all primes $\mathfrak{p}$ in $K$ above rational primes $p$
        of norm greater than $C(L,K)$, the image of the mod $p$
        representation associated to $A$ contains
        $\mathrm{SL}_2(\F_{\mathfrak{p}})$, where $\F_{\mathfrak{p}}$
        denotes the residue field of $K$ at $\mathfrak{p}$.
\end{conj}
We can prove that any non-trivial solution corresponds to a
$\GL_2$-type abelian variety without complex multiplication. Assuming
Conjecture~\ref{conj:large-image} we get the following stronger result.
{\color{MidnightBlue}
\begin{thmC}\label{thmC}
  \color{black} Let $\bfP$ be as in Theorem \ref{thmA}. Assume that
  Conjecture~\ref{conj:large-image} holds for $L=K=\Q(\sqrt{5})$ with
  constant $C$. Then if $p$ is a prime number satisfying that
  $p \not \in \bfP$ and $p > C$, there
  are no non-trivial primitive solutions $(a,b,c)$ to
 \[
 x^5 + y^p + z^3 =0
 \]
 with $3\nmid c$.
\end{thmC}
}
The condition $3 \nmid c$ is used to eliminate the two forms coming
from ghost solutions. We can prove (see
Theorem~\ref{thm:local-type-3}) that when $3 \nmid c$ the local type
of a solution does not match the one of a ghost solution. Then Theorem \ref{thmB} follows from Theorem \ref{thmA}.

To the best of our knowledge, Theorems~\ref{thmB} and \ref{thmC} represent the first
Diophantine result for an infinite family of signatures involving
three distinct prime exponents. It also constitutes the first
successful application of the hypergeometric-motives approach within
the modular method.
\subsection*{Notation} Throughout this paper we use the following notation. Let $F$ denote a number field. 
\begin{itemize}
  
\item We denote by $\Gal_F$ the absolute
  Galois group $\Gal(\overline{F}/F)$.
\item If $c\in F$, we denote by $\theta_c$  the character of $\Gal_F$ corresponding by class field theory
  to the extension $F(\sqrt{c})/F$.
\item If $\id{n}$ is a prime ideal of $F$, we denote by $\normid{n}$ its norm.
\item If $\newform$ is a newform, we denote by $K_\newform$ its coefficient field.
  
\end{itemize}

\subsection*{Electronic resources.} The \texttt{Magma} \cite{Magma} and \texttt{PARI/GP} \cite{PARI2} programs used for the computations in this article are posted in \cite{github}, which also includes a list of the programs, their descriptions, output transcripts, timings, and the machines used.

All calculations were performed using the servers of the Math Computation cluster shanks-birch @ CMat (Montevideo, Uruguay), funded by Pedeciba, CSIC I+D and ANII $FCE\_1\_2017$ $\_1\_136609$.

\subsection*{Acknowledgments.}
The present article has benefited from many fruitful discussions with colleagues. We are grateful to Fernando Rodriguez Villegas and Franco Golfieri Madriaga for valuable conversations on hypergeometric motives and their Diophantine applications. We also thank Imin Chen for insightful suggestions related to the subject of this work, and Tim Dokchitser for sharing his code to compute cluster pictures. Our thanks go as well to Nuno Freitas for discussions concerning the computation of an equation for the local type at 3 of our hyperelliptic curve. Finally, we are indebted to John Jones for explaining the use of his implementation of Panayi's algorithm for computing reductions of local extensions. 

\section{Hypergeometric motives}
\label{section:prel}
The foundations of the theory of hypergeometric motives is due to Katz
(\cite{katz,MR1366651}); we recommend the reader to look at the nice
survey \cite{MR4442789} for a friendly introduction to rational
hypergeometric motives. For Diophantine applications, we are mostly
interested in rank $2$ motives which are defined over totally real
number fields, as studied in \cite{GPV}. The goal of the present
section is not to give a general review of the theory, but to present
a concise account of its main properties (whose proofs are given in \cite{GPV}).

Let us consider the generalized Fermat equation
\begin{equation}\tag{\ref{eq:GFE}}
	x^q+y^p+z^r=0,
\end{equation}
where $q\neq r$ and $\min\{q,r,p\}>2$. Without loss of generality, we
assume that $q>r$. Darmon's innovative idea (developed in
\cite{Darmon}) is to attach a varying family of residual
representations, called \emph{Frey representations}, to the signature $(q,p,r)$, such that each particular solution 
$(a,b,c)$ to (\ref{eq:GFE}) gives rise to a specialization of this family.  In \cite{GP} the authors applied a
similar strategy, replacing Darmon's Frey representations with
hypergeometric motives. More concretely, to the signature $(q,p,r)$ they
attach the two hypergeometric motives
\begin{align*}
  \HGM ^+(t):=\HGM\left(\left(\frac{1}{r},-\frac{1}{r}\right),\left(\frac{1}{q},-\frac{1}{q}\right)\Big | t\right),  \quad \HGM^-(t):=\HGM\left(\left(\frac{1}{2r},-\frac{1}{2r}\right),\left(\frac{1}{q},-\frac{1}{q}\right)\Big | t\right),
\end{align*}
for $t$ the varying parameter.
\subsection{Geometric properties}
Rank $2$ hypergeometric motives are a motivic incarnation of Gauss
hypergeometric function $_{2}F_1(a,b;c | t)$. The function satisfies an
ordinary differential equation, so has attached a \emph{monodromy
  representation}, i.e.\
$\rho:\pi_1(\PP^1\setminus\{0,1,\infty\},x_0) \to \GL_2(\CC)$. If
$\gamma_i$ denotes a loop in counterclockwise direction around the
point $i$, for $i=0,1,\infty$, then the monodromy representation for
$\HGM^+(t)$ and $\HGM^-(t)$ at $\gamma_i$ has matrices conjugate to
\begin{equation}
  \label{eq:monodromy+}
M_0^+:=
\begin{pmatrix}
  \zeta_q^s & 0\\
  0 & \zeta_q^{-s}
\end{pmatrix},
\qquad
M_1^+:=
\begin{pmatrix}
  1 & 1\\
  0 & 1
\end{pmatrix},
\qquad
M_\infty^+:=
\begin{pmatrix}
  \zeta_r & 0\\
  0 & \zeta_r^{-1}
\end{pmatrix},
\end{equation}
and
\begin{equation}
  \label{eq:monodromy-}
M_0^-:=
\begin{pmatrix}
  -\zeta_q^s & 0\\
  0 & -\zeta_q^{-s}
\end{pmatrix},
\qquad
M_1^-:=
\begin{pmatrix}
  1 & 1\\
  0 & 1
\end{pmatrix},
\qquad
M_\infty^-:=
\begin{pmatrix}
  \zeta_r & 0\\
  0 & \zeta_r^{-1}
\end{pmatrix},
\end{equation}
respectively, where $\zeta_r:=\exp(\frac{2 \pi i}{r})$. A result of
Levelt proves that one can chose a basis so that the representation in
fact has image in $\GL_2(\Q(\zeta_{qr}))$.

For every  prime ideal $\id{p}$ of $K$, the monodromy representation extends to
a continuous representation
\begin{equation}
  \label{eq:geom-rep}
\rho_{\id{p}}^{\pm}:\Gal_{\overline{\Q}(t)} \to \GL_2(K_{\id{p}}),
\end{equation}
with the property that the image of the inertia group around the point
$0, 1, \infty$ is generated (up to conjugation) by the matrix
$M_0^\pm, M_1^\pm$ and $M_{\infty}^\pm$ respectively.

For $N$ a positive integer, let $\Q(\zeta_N)^+$ denote the totally
real subfield of the cyclotomic extension $\Q(\zeta_N$).  The motives
$\HGM^{\pm}(t)$ are defined over the totally real number field
$K:=\Q(\zeta_q)^+\cdot \Q(\zeta_r)^+$ (the composition of both
fields). Then they have attached Galois representations
\[
\rho_{\id{p}}^{\pm}: \Gal_{K(t)} \to \GL_2(K_{\id{p}})
\]
that extend the representations (\ref{eq:geom-rep}) (a
property that characterizes the motive up to a twist).

\subsection{Construction of the motive}
Let us recall an explicit construction of the motives that will be
useful for later purposes. Set
\[
  A^+:=q-r, \quad B^+:=-q-r, \quad C^+:=q+r, \quad D^+:=-r.
\]
Consider Euler's curve given by the equation
\begin{equation}
  \label{eq:Euler+}
\bfC^+: y^{qr} = x^{A^+}(1-x)^{B^+} (1-tx)^{C^+}t^{D^+}.    
\end{equation}
The curve $\bfC^+$ has an action of the group $\mubb_{qr}$ of $qr$-th
roots of unity (via $\zeta_{qr}\cdot(x,y)=(x,\zeta_{qr}y)$). Over the
cyclotomic field $F=\Q(\zeta_{qr})$, the motive $\HGM^+(t)$
corresponds to the $\zeta_{qr}$-eigenspace for this action on the new
part of $\Jac(\bfC^+)$.

To define the motive over $K$ we need to consider quotients of
$\bfC^+$.  The curve possesses three involutions denoted by
$\iota_{-1}$, $\iota_j$ and $\iota_i$,  defined as follows:
\begin{equation}
  \label{eq:inv-1}
  \iota_{-1}(x,y)=\left(\frac{1}{tx},\frac{1}{y}\right).
\end{equation}
Let $j$ be an integer
satisfying that $j \equiv 1 \pmod r$ and $j \equiv -1 \pmod q$. Let
$\mu, \nu$ be positive integers so that
\[
\frac{j}{r} = \frac{1}{r} + \mu,\qquad \frac{j}{q} = -\frac{1}{q} + \nu.
\]
Then by \cite[Lemma 8.10]{GPV} the map
\begin{equation}
  \label{eq:inv1}
  \iota_j(x,y)=\left(\frac{x-1}{tx-1},\frac{x^{\mu-\nu}(1-x)^{-\mu-\nu}(1-xt)^{\mu+\nu}t^{-\nu}}{y^j}\right)  
\end{equation}
is an involution. There is a similar definition for their composition:
let $i$ be an integer satisfying $i \equiv 1 \pmod r$ and $i \equiv -1 \pmod q$. Let $\mu:= \frac{i-1}{r}$,  $\nu:= \frac{i+1}{q}$. 
Then by \cite[Lemma 8.11]{GPV},
the map
\begin{equation}
  \label{eq:inv-2}
  \iota_i(x,y)=\left(\frac{tx-1}{t(x-1)},\frac{x^{\mu-\nu}(1-x)^{-\mu-\nu}(1-tx)^{\mu + \nu}t^{-\nu}}{y^i}\right)
\end{equation}
is an involution corresponding to their composition. Let $\bfD^+$ be
the quotient of $\bfC^+$ by $\langle \iota_{-1},\iota_j\rangle$. Then
the \textit{new} part of $\Jac(\bfD^+)$ has an action of $\Om_K$, the
ring of integers of $K$, and it is of $\GL_2$-type over $K$ (see
\cite[Proposition 8.12]{GPV}). A Tate twist of its new part gives a
realization of the motive $\HGM^+(t)$ over $K$.

A similar construction holds for the motive $\HGM^-(t)$. Set
\[
  A^-:=q-2r, \quad B^-:=-q-2r, \quad C^-:=q+2r, \quad D^-:=-2r.
\]
and define Euler's curve
\begin{equation}
  \label{eq:Euler-}
\bfC^-: y^{2pq} = x^{A^-}(1-x)^{B^-} (1-tx)^{C^-}t^{D^-}.    
\end{equation}
The same quotient as before produces a realization of the motive
$\HGM^-(t)$.

\subsection{Specializations}
Let $(a,b,c)$ be a putative primitive solution to (\ref{eq:GFE}), and
set $t_0=-\frac{a^q}{c^r}$. We will use the following notation:
\begin{equation}
  \label{eq:specialzation}
\mot^+:=\HGM^+(t_0),\qquad \mot^-:=\HGM^-(t_0)\otimes \theta_c,
\end{equation}
where $\theta_c$ is the quadratic character of $\Gal_\Q$ corresponding to
the extension $\Q(\sqrt{c})/\Q$ (a similar notation will be used for
its restriction to $\Gal_K$). The specialized motive $\mot^\pm$
satisfies that its coefficient field (i.e.\ the field generated by
Frobenius traces) is contained in $K$, the field of definition of the motive \cite[Theorem 8.2]{GPV}.
Then for each prime
ideal $\id{p}$ of $K$ there exists a Galois representation
\begin{equation}
  \label{eq:gal-rep}
  \rho_{\id{p}}^\pm:\Gal_K \to \GL_2(K_{\id{p}})  
\end{equation}
attached to the motive $\mot^\pm$. The following are some properties of $\mot^+$.
\begin{thm}
  \label{thm:mot-+-prop}
  Let $\idK$ be a prime ideal of $K$. Then
  \begin{enumerate}
  \item The motive $\mot^+$ is unramified at $\idK$ if $\idK \nmid qrb$.
    
  \item It has multiplicative reduction at primes
    $\idK \nmid qr$ dividing $b$.
  \item The residual representation of $\rho_{\id{p}}^+$ is unramified
    at all primes not dividing $pqr$.
    
  \item The residual representation of $\rho_{\id{p}}^+$ is finite at
    all primes dividing $p$.
    
  \item If either $q \mid a$ or $r \mid c$ then $\mot^+$ is modular.
  \end{enumerate}
\end{thm}
\begin{proof}
  See \cite{GP}, Theorems 4.1, 4.2, 4.3 and 5.5 respectively.
\end{proof}
\begin{remark}
  \label{rem:Tate-twist} As explained in \cite[Remark 6]{GP}, the
  modularity statement is true up to a Tate twist. For the parameters
  studied in the present article, $\mot^{\pm}(1)$ has Hodge-Tate
  weights matching a parallel weight $2$ Hilbert modular form.
\end{remark}
The motive $\mot^-$ satisfies analogous properties.
\begin{thm}
  \label{thm:mot---prop}
  Let $\idK$ be a prime ideal of $K$. Then
  \begin{enumerate}
  \item The motive $\mot^-$ is unramified at $\idK$ if $\idK \nmid 2qrb$.
  \item It has multiplicative reduction at primes $\idK \nmid 2qr$
    dividing $b$.
    
  \item The residual representation of $\rho_{\id{p}}^-$ is unramified
    at all primes not dividing $2pqr$.
    
  \item The residual representation of $\rho_{\id{p}}^-$ is finite at
    all primes dividing $p$.
    
  \item If either $r>5$ and $q \mid a$ or $q>5$ and $r \mid c$ then
    $\mot^-$ is modular.
  \end{enumerate}
\end{thm}
\begin{proof}
  See \cite{GP}, Theorems 4.1, 4.2, 4.3 and 5.6 respectively.
\end{proof}
When $r=3$ we can improve the previous modularity result.
\begin{thm}
  \label{thm:modularity}
  Let $r=3$. Then
  \begin{enumerate}
  \item If $q\ge5$, the motive $\mot^+$ is modular for all
    specialization of the parameter $t$.
  \item If $q\ge 11$, the motive $\mot^-$ is modular for all
    specialization of the parameter $t$.
  \end{enumerate}
\end{thm}
\begin{proof}
  See \cite[Theorem 6.3]{GP}.
\end{proof}
\begin{remark}
  Unfortunately the strategy used in \cite{GP} (suggested in
  \cite{Darmon}) does not allow to prove modularity of all specializations of the motive
  $\mot^-$ when $q=5$. The reason is that for some specializations of
  the parameter, the representation modulo $5$ is reducible. Let us
  explain the situation with more detail.  By \cite[Theorem
  10.3]{GPV}, the motive
  $\HGM\left(\left(\frac{1}{6},-\frac{1}{6}\right),\left(\frac{1}{5},-\frac{1}{5}\right)|t\right)$
  is congruent modulo $(\sqrt{5})$ to the rational motive
  $\HGM\left(\left(\frac{1}{6},-\frac{1}{6}\right),\left(1,1\right)|t\right)$.
  The latter corresponds to the rational elliptic curve (see \cite[\S
  2.1.4 ]{GP})
  \[
    E_t: y^2+xy=x^3-\frac{t}{432}.
  \]
  The curve is modular for all rational specializations $t_0$ of the
  parameter, so if its residual representation modulo $5$ is
  irreducible, the motive
  $\HGM\left(\left(\frac{1}{6},-\frac{1}{6}\right),\left(\frac{s}{5},-\frac{s}{5}\right)|t_0\right)$
  is modular as well. The problem is that the residual representation
  of $E_{t_0}$ modulo $5$ is not always irreducible. For example,
  taking $t_0=\frac{27}{16}$, $E_{t_0}$ matches the elliptic curve with
  LMFDB label \lmfdbec{176}{b}{3}, which has a rational
  $5$-subgroup. Similarly, the value $t_0=-\frac{11}{16}$ corresponds
  to the elliptic curve with label \lmfdbec{11}{a}{3}, which has a
  $5$-rational point.
\end{remark}
\subsection{Traces of Frobenius elements}
Let $N$ be a positive integer (in our case $N^+=qr$ and $N^-=2qr$). Let $\idF$ be a prime ideal of
$F=\Q(\zeta_{N})$ not dividing $N$ with residue field $\F_{\ell^r}$. Following~\cite{MR0051263},
define the character
\begin{equation}
  \label{eq:char-def}
  \chi_{\idF}: (\ZZ[\zeta_N]/\idF)^{\times} \rightarrow
  \CC^{\times}
\end{equation}
as follows: for $x$ an integer prime to $\idF$, let
$\chi_{\idF}(x)$ be the $N$-th root of unity congruent to
$x^{(\ell^r-1)/N}$ modulo $\idF$. If $\idF \mid x$ set
$\chi_{\idF}(x)=0$ . By construction, the character $\chi_{\idF}$ has order $N$.

Fix an additive character $\psi$ on $\F_{\ell^r}$. For
$\omega \in \widehat{\F_{\ell^r}^\times}$, a character of $\F_{\ell^r}^\times$,
denote by $g(\psi,\omega)$ the Gauss sum
\begin{equation}
\label{gsum}
g(\psi,\omega) = \sum_{x \in \F_{\ell^r}^\times} \omega(x)\psi(x).
\end{equation}
%
\begin{defi}[Finite hypergeometric sum]
  \label{defi:finite-hgs}
  For $a,b,c,d \in \Z[\frac{1}{N}]$ and $t_0 \in \F_{\ell^r}$, define
  the finite hypergeometric series
\[
H_{\idF}((a,b),(c,d)|t_0)=\frac{1}{1-\ell^r} \sum_{\omega \in \widehat{\F_{\ell^r}^\times}} 
\frac{g(\psi,\chi_{\idF}^{-aN}\omega) g(\psi,\chi_{\idF}^{cN} \omega^{-1})}{g(\psi,\chi_{\idF}^{-aN})
	g(\psi,\chi_{\idF}^{cN})}\frac{g(\psi,\chi_{\idF}^{-bN}\omega) g(\psi,\chi_{\idF}^{dN}\omega^{-1})}{g(\psi,\chi_{\idF}^{-bN})
	g(\psi,\chi_{\idF}^{dN})}\omega(t_0).
\]
\end{defi}
It is not hard to verify that the definition does not depend on the choice of the additive
character $\psi$. Set:
%
\begin{align}	\label{eq:N-pm}
	&a^+=\frac{1}{r}, \quad b^+=-\frac{1}{r}, \quad c=\frac{1}{q}, \quad d=-\frac{1}{q}, \quad N^+=qr.\\ \notag
	&	 a^-=\frac{1}{2r}, \quad b^-=-\frac{1}{2r}, \quad  N^-=2qr. 
\end{align}
For $\id{p}$ a prime ideal of $F$, denote also by
$\rho_{\id{p}}^{\pm}$ the restriction of (\ref{eq:gal-rep}) to
$\Gal_F$.
\begin{thm}
  \label{thm:trace-frob}
  Let $t_0$ be a rational number. Let $\id{p}$ be a prime ideal of $F$
  and let $\idF$ be a prime ideal of $F$ not dividing $\normid{p}\cdot N^{\pm}$. Let
  $\Frob_{\idF}$ be a Frobenius element of $\Gal_F$. Then if
  $v_{\idF}(t_0(t_0-1))=0$, the representation $\rho_{\id{p}}^{\pm}$
  is unramified at $\idF$ and
    \[
      \trace{\rho_{\id{p}}^{\pm}(\Frob_{\idF})} = H_{\idF}((a^{\pm},b^{\pm}),(c,d)|t_0).      
\]
\end{thm}
\begin{proof}
  See \cite[Theorems 4.8 and 7.23]{GPV}.
\end{proof}

\section{Conductor at potentially wild primes}
\label{sec:conductor}
We begin by recalling some basic properties of the conductor of a Galois representation.  Let $K$ be a number field and let
$\rho : \Gal_K \to \GL_2(\overline{\Q_p})$ be a Galois
representation. Let $\id{q}$ be a prime ideal of $K$ whose residual
characteristic is prime to $p$ and let $\id{n}_{\id{q}}$ denote the
$\id{q}$-th valuation of Artin's conductor of $\rho$. Then
\[
  \mathfrak{n}_{\id{q}} = \mathfrak{n}_{\id{q}}^{\text{tame}} +
  \mathfrak{n}_{\id{q}}^{\text{wild}} ,\] where
$\mathfrak{n}_{\id{q}}^{\text{tame}}$, called the \textit{tame} conductor,  is the codimension of the
subspace fixed by the inertia group $I_{\id{q}}$, and
$\mathfrak{n}_{\id{q}}^{\text{wild}}$, the \emph{wild} or \emph{Swan} 
conductor, is the sum of the codimensions of the spaces fixed by the
higher ramification groups (see for example Chapter VI, \S 2 of
\cite{MR554237}). It is easy to verify that if
$\id{n}_{\id{q}}^{\text{wild}} \neq 0$ then
$\id{n}_{\id{q}}^{\text{tame}} \neq 0$.
%
%

The \emph{potentially wild primes} for the motive $\mot^{\pm}$
(i.e.\ the prime ideals $\idK$ of $K$ for which
$\id{n}_{\idK}^{\text{wild}}$ might be non-zero) are those prime
ideals of $K=\Q(\zeta_q)^+\cdot \Q(\zeta_r)^+$ dividing $N^{\pm}$,
i.e.\ for $\mot^+$ is the set of primes dividing $qr$ while for
$\mot^-$ is the set of primes dividing $2qr$ (see (\ref{eq:N-pm})).

 The goal of the present
section is to compute the conductor of the motives $\mot^\pm$ at
potentially wild primes. In doing so, we follow the strategy used in
\cite{GP} and \cite[Section 6.4]{PedroLucas}, namely relate our motive
to a hyperelliptic curve (as defined by Darmon in \cite{Darmon} and
whose definition we recall) and use the theory of clusters (as
developed in \cite{MR4566695}) to compute its local conductor.

\vspace{7pt}

Let $h(x)$ denote the minimal polynomial of $\zeta_r+\zeta_r^{-1}$ and
consider the odd function
\begin{equation}
  \label{eq:f-pol}
f(x)=(-1)^\frac{r-1}{2}xh(2-x^2).   
\end{equation}
For $t$ a variable, consider the following two curves defined in \cite[Equations (3) and (4)]{Darmon}.
\begin{equation}
  \label{eq:Darmon-hyper}
C_r^-(t) : y^2 = f(x) + 2 - 4t,\qquad C_r^+(t) : y^2 =(x+2) (f(x) + 2 - 4t).  
\end{equation}
\begin{prop}
\label{prop:darmon-hyper-relation}
  Let $t_0 \in \Q\setminus\{0,1\}$. Then
  \begin{itemize}
  \item The hypergeometric motive
    $\HGM((\frac{1}{r},-\frac{1}{r}),(1,1)|t_0)$ is isomorphic to a rank $2$
    motive of the curve $C_r^+(t_0)$.
    
  \item The hypergeometric motive
    $\HGM((\frac{1}{2r},-\frac{1}{2r}),(1,1)|t_0)$  is isomorphic to the
    quadratic twist by $\theta_{-1}$ of a rank $2$ motive of the curve
    $C_r^-(t_0)$.
  \end{itemize}
\end{prop}
\begin{proof}
  See Corollaries A.5 and A.11 of \cite{GP}.
\end{proof}
For later purposes, let us recall the following congruences between hypergeometric motives.
\begin{prop}
  \label{prop:congruences}
  Let $\id{q}$ be a prime ideal of $F=\Q(\zeta_{N^\pm})$  dividing $q$. Let 
  $t_0 \in \Q\setminus\{0,1\}$. Then the following congruences hold.
  \begin{itemize}
  \item The motive $\HGM^+(t_0) \equiv \HGM((\frac{1}{r},-\frac{1}{r}),(1,1)|t_0)\pmod{\id{q}}$.
  \item The motive $\HGM^-(t_0) \equiv \HGM((\frac{1}{2r},-\frac{1}{2r}),(1,1)|t_0)\pmod{\id{q}}$.
  \end{itemize}
\end{prop}
\begin{proof}
  See   \cite[Theorem 10.3]{GPV}.
\end{proof}
\begin{remark}
  \label{remark:twist}
  The previous congruences only hold over the field $F$. Since the
  extension $F/K$ is biquadratic, we may deduce that, over $K$, a
  similar congruence holds up to a quadratic twist. Furthermore, the
  twist is independent of the parameter $t$ (see \cite[Corollary
  2.2]{2502.02776}), so it can be easily computed for each signature (see Lemma~\ref{lemma:congruences}).
\end{remark}

\subsection{Local inertial type}

The last two propositions suggest that we should start studying the
reduction type of Darmon's curves $C_r^{\pm}(t)$ at primes dividing
$r$ and its ``local inertial type''. 
 Let $K/\Q_\ell$ be a local
field, and let $\rho : \Gal_K \to \GL_2(\overline{\Q_p})$ be a Galois
representation, with $p \neq \ell$.

\begin{defi}\label{def:local-type}
  The \emph{local inertial type} (or \emph{local type} for short) of $\rho$ is the isomorphism class of the
  restriction of $\rho$ to the inertia subgroup.
\end{defi}
To characterize local types, it is customary to study representations
of the Weil-Deligne group. Instead of recalling the definition of the
Weil-Deligne group (see \S 4.1 of \cite{MR546607}) we recall the
definition of its representations. A Weil-Deligne representation
consists of a 2-dimensional complex representation $\rho$ of the
Weil group $W_K$ together with a monodromy operator $N \in M_2(\CC)$. Let
$\omega_1: W_{\Q_\ell} \to \CC^\times$ be the unramified character
sending Frobenius to $\|\ell\|_\ell$. For $\ell \neq 2$, there are
three different local types for a Weil-Deligne representation $\rho$
with trivial Nebentypus:

\begin{enumerate}
\item {\bf Principal Series}: $N =\left(\begin{smallmatrix} 0 & 0\\ 0 & 0\end{smallmatrix}\right)$ and
  $\rho = \chi \oplus \chi^{-1}\omega_1^{1-k}$ for some
  quasi-character $\chi:W(\QQ_p)^{\text{ab}} \to \CC^\times$.
\item{\bf Special}:
  $N =\left(\begin{smallmatrix}0 & 1\\ 0 & 0 \end{smallmatrix}
  \right)$ and the representation $\rho$ equals
  $\omega_1^r \left(\begin{smallmatrix}\chi \omega_1 & 0\\
                      0 & \chi\end{smallmatrix} \right)$ for some
  quasi-character $\chi:W_K \to \CC^\times$.
\item{\bf Supercuspidal}: $N=\left(\begin{smallmatrix} 0 & 0\\ 0 & 0\end{smallmatrix}\right)$ and $\rho=\Ind_{W_E}^{W_K}\varkappa$,
  where $E$ is a quadratic extension of $K$, and
  $\varkappa:W_E^{\text{ab}} \to \CC^\times$ is a quasi-character which does
  not factor through the norm map with a quasi-character of
  $W_K^{\text{ab}}$.
\end{enumerate}
In many instances, the local type is preserved under a congruence between
two Galois representations (see e.g.\ \cite[Proposition
1.1]{MR4583916}).


\subsection{Conductor at primes dividing $r$}
For the reader's convenience, we start with the case $r=3$, since
in this case the curves have genus $1$ (the local theory of elliptic
curves is better understood).
\subsubsection{The case $r=3$} 
The curve $C_3^+$ and the quadratic twist by $-1$ of
the curve $C_3^-$ are elliptic curves given by the equations
\begin{equation}
  \label{eq:r=3-curves}
  E_3^+(t) \colon y^2+3xy+ty=x^3, \qquad E_3^-(t) \colon y^2=x^3-3x+4t-2. 
\end{equation}
The discriminant and $j$-invariant of each curve are
\begin{align*}
	& \Delta(E_3^+(t))=-3^{3}t^3(t-1), &  j(E_3^+(t))=\frac{3^3(8t-9)^3}{t^3(t-1)},\\
	&\Delta(E_3^-(t))=-2^83^3t(t-1),
	 &   j(E_3^-(t))=-\frac{2^43^3}{t(t-1)}.
\end{align*}
%
	Following standard notations, when the curve has potential good
	reduction, we denote by $e$ the degree of the minimal extension of
	$\Q_p^{\text{ur}}$ (the maximal unramified extension of $\Q_p$) where
	it acquires good reduction.

\begin{prop}
  \label{prop:Et-reduction}
  Let $t_0 \in \Q\setminus\{0,1\}$. Then Table~\ref{table:local-3} contains the conductor exponent and the local type of the elliptic
  curve $E_3^\pm(t_0)$ at the prime $p=3$ . When $v_3(t_0)<0$, we assume that
  $3 \mid v_3(t_0)$ and we use the notation $t_0=3^{v_3(t_0)}t_0'$.
\begin{table}[H]
\begin{tabular}{|c||c|c|c|c|}
  \hline
  Condition & $v_3(\cond(E_3^+(t_0)))$ & $v_3(\cond(E_3^-(t_0)))$& $e$ & Type\\
  \hline
  \hline
  $v_3(t_0) > 3 $ & $1$ & $2$ & -- & special\\
  \hline
  $v_3(t_0-1) > 3$ & $2$ & $2$ &--  & special \\
  \hline
  $t_0 \equiv 5 \pmod 9$ & $2$ & $2$ & $4$ & supercuspidal\\
  \hline
  $t_0 \equiv 2,8 \pmod 9$ & $3$ & $3$ & $12$ & supercuspidal\\
  \hline
  $t_0'\equiv \pm 2 \pmod 9$ & $2$ & $2$ & $4$ & supercuspidal\\
  \hline
  $t_0'\not \equiv \pm 2 \pmod 9$ & $3$ & $3$ & $12$ & supercuspidal\\
  \hline
\end{tabular}
\caption{Local information of  $E_3^{\pm}(t_0)$ at $3$.}
\label{table:local-3}
\end{table}
\end{prop}

\begin{proof}
  In the first two cases, the $j$-invariant of the curve has negative
  valuation at $3$, so we just need to determine whether the curve has
  additive or multiplicative reduction at $3$.
  For $E_3^-(t_0)$ it is additive and for $E_3^+(t_0)$ is
  multiplicative in the first case and additive in the second one. To
  compute the third entry, Tate's algorithm shows that the curve has
  Kodaira type III at $3$, so $e=4$ by \cite[Th\'eor\`eme
  1]{Kraus}. The local type can be computed using \cite[Table
  1]{DFV}. The Kodaira type for the fourth row is II, hence by
  \cite[Th\'eor\`eme 1]{Kraus} $e=12$. The last two statements follow
  from a similar argument.
\end{proof}

\begin{coro}
  \label{coro:conductor-at-3}  
  Let $q>3$ be a prime number. Let $K=\Q(\zeta_q)^+$ and let $\id{r}$
  be a prime ideal of $K$ dividing $3$.  Let $(a,b,c)$ be a
  non-trivial solution to $(\ref{eq:GFE})$ with $p>2$. Set
  $t_0=-\frac{a^q}{c^3}$. Then the possible values for the conductor
  exponent of $\mot^\pm$ at $\id{r}$ are given in
  Table~\ref{table:local-3} (when $3 \mid c$, we use the notation
  $c = 3^{v_3(c)}c_0$).
  \begin{table}[H]
    \begin{tabular}{|c|r||c|r|}
      \hline
      $v_{\id{r}}(\cond(\mot^\pm))$ &Condition \qquad \qquad& $v_{\id{r}}(\cond(\mot^\pm))$ &Condition \qquad \qquad\qquad\\
      \hline
      \hline
      $1$, $2$ & $3\mid a$\qquad \qquad  & $0$, $1$, $2$ & $3 \mid b$ \qquad \qquad \qquad\\
      \hline
      $2$ & $4c^3 \equiv a^q \pmod 9$ & $2$ & $3\mid c$, $\pm 2c_0^3 \equiv a^q \pmod 9$\\
      \hline
      $3$ & $c^3 \equiv a^q \pmod 9$  & $3$ & $3\mid c$, $\, \, \, \pm c_0^3 \equiv a^q \pmod 9$\\
      \hline
      $3$ & $7c^3 \equiv a^q \pmod 9$ &$3$ & $3\mid c$, $\pm 4c_0^3 \equiv a^q \pmod 9$ \\
      \hline
    \end{tabular}
    \caption{Conductor exponent of $\mot^{\pm}$ at $3$.}
    \label{table:cond3}
\end{table}
\end{coro}
\begin{proof}
  The field $K=\Q(\zeta_q)^+$ is the field of definition of the motive
  $\mot^{\pm}$, while $E_3^{\pm}(t_0)$ is defined over $\Q$.  Let
  $\rho_{E_3^{\pm},q}:\Gal_\Q \to \GL_2(\ZZ_q)$ be the $q$-adic Galois
  representations attached to $E_3^\pm :=E_3^\pm(t_0)$. 
  
  Start considering the values $t_0$ where
  the local type of $E_3^\pm$ is supercuspidal at $3$ (as
  detailed in Table~\ref{table:local-3}). Since $q> 3$, the image of a
  decomposition group $D_3$ for the residual representation
  $\bar{\rho}_{E_3^{\pm},q}$ is also dihedral, so its residual
  tame conductor equals $2$. Since $K/\Q$ is unramified at $3$, the
  same is true for the residual representation
  $\bar{\rho}_{E_3^{\pm},q}|_{\Gal_K}$.
  
  Let $\id{q}$ be a prime in $K$ above $q$.  By Propositions
  \ref{prop:darmon-hyper-relation} and \ref{prop:congruences}, and
  Remark \ref{remark:twist}, the $\id{q}$-adic residual representation
  of $\mot^\pm$ is isomorphic to
  $\bar{\rho}_{{E_3^\pm},q}|_{\Gal_K}$, up to a quadratic
  twist. Hence, its tame conductor is also $2$. The result then
  follows from the fact that the Swan conductor is preserved under
  congruences (see e.g.\ \cite[Proposition 3.1.42]{Wise}).
	
  When $3 \mid ab$ the curve $E_3^{\pm}$ has (potentially)
  multiplicative reduction at $3$, hence its Swan conductor is $0$, so
  the conductor exponent of $\mot^{\pm}$ is at most $2$.
%
%
%
Recall that $t_0 = -\frac{a^q}{c^r}$, so if $3\mid a$, 
  \[
    v_3(\Disc(E_3^{\pm}))\equiv3\not\equiv 0\pmod q,
  \]
  so the residual representation $\bar{\rho}_{E_3^{\pm},q}$ is
  ramified at $3$ (by the theory of Tate's curves), giving a lower
  bound of $1$ to the tame conductor of $\mot^{\pm}$.
\end{proof} 
\begin{remark}
  \label{remark:non-twist}
  A little more can be said when the congruences of
  Proposition~\ref{prop:congruences} hold over $K$ (i.e.\ no twist is
  needed, see Remark \ref{remark:twist}). If $E_3^{\pm}(t_0)$ is special with conductor
  valuation $2$, its residual conductor also has valuation $2$ (since
  the local type is special with the character $\chi$ quadratic and
  ramified) hence in this case $v_{\id{r}}(\cond(\mot^\pm))=2$.
\end{remark}
\begin{remark}
  \label{remark:potentially-good}
  The isomorphism of residual representations between $E_3^{\pm} :=E_3^\pm(t_0)$
  and $\mot^{\pm}$ allows us to compute the conductor at $\id{r}$ of
  $\mot^\pm$ from that of the elliptic curve $E_3^{\pm}$, as done in
  the Corollary \ref{coro:conductor-at-3}. However, this is not
  enough to compute its local type, since even when the field
  extension $K/\Q$ is unramified at $3$, the local type is not
  necessarily preserved under restriction.  Looking at
  Table~\ref{table:local-3}, we get:
  \begin{itemize}
  \item If the local type is special, it is indeed preserved under base extension.
    
  \item When the local type is supercuspidal and $e=12$, the local
    representation at $3$ is induced from a quadratic ramified
    extension of $\Q_3$ (see e.g.\ \cite[Corollary
    3.1]{MR3056552}), hence the local type is indeed preserved.
    
  \item When the local type is supercuspidal and $e=4$, the local type
    corresponds to the induction of an order $4$ character from
    $\Q_9$ (the unramified quadratic extension of $\Q_3$) \cite[Proposition 4.2.1]{DFV}. If
    $K_{\id{r}}$ contains $\Q_9$
  then the local type of the restriction is a ramified principal
  series; otherwise, it is a supercuspidal representation.  Note that
  $K_{\id{r}}/\Q_3$ is unramified of order equal to $\ord_q(3^2)$,
  hence contains $\Q_9$ precisely when $\ord_q(3)$ is divisible by $4$.
  \end{itemize}
  Even when the local type of $\mot^{\pm}$ at $\id{r}$ is not
  determined by the congruence, it is still the case that when
  $3 \nmid ab$ it has potentially good reduction at $\id{r}$.
\end{remark}

\subsubsection{The case $r\ge5$}
\label{section:r-ge-5}
Let $(a,b,c)$ be a non-trivial solution to (\ref{eq:GFE}) and set
$t_0=-\frac{a^q}{c^r}$.  Let as before $K=\Q(\zeta_q)^+$ and let $\id{r}$
be a prime ideal of $K$ dividing $r$. We need to understand the
semistable reduction of the curves $C_r^\pm(t_0)$ in
(\ref{eq:Darmon-hyper}) over $K_{\id{r}}$. To get an integral equation, define
\begin{equation}\label{eq:Def-F}
	F(x)=c^rf\Big(\frac{x}{c}\Big)+2c^r + 4a^q,
\end{equation}
for $f(x)$ as in (\ref{eq:f-pol}), and consider the two
hyperelliptic curves
\begin{align*}
	&C_r^-(a,b,c) : y^2 = F(x),\\
	&C_r^+(a,b,c) : y^2 =(x+2c)F(x). 
\end{align*}

\begin{lemma}
  \label{lemma:iso-curves}
  The following relations hold:
  \begin{enumerate}
  \item $C_r^-(a,b,c)$ is isomorphic to $C_r^-(t_0)$ over
    $\Q(\sqrt{c})$.
  \item $C_r^+(a,b,c)$ is isomorphic to $C_r^+(t_0)$ over $\Q$.
  \end{enumerate}
  \begin{proof}
    The first statement follows from the change of variables
    $(x,y) \to (x/c,y/\sqrt{c^r})$, while the second one by
    $(x,y) \to (x/c,y/c^{\frac{r+1}{2}})$.
  \end{proof}
\end{lemma}
Let $J^\pm_r(t_0)$
be the Jacobians of $C_r^\pm(t_0)$
over $K$. By \cite{TTV}, the abelian variety $J_r^\pm(t_0)$ is of
$\GL_2$-type over $K$. Therefore, it has an associated compatible
system of $2$-dimensional Galois representations
$\{\rho_{J_r^\pm (t_0),\idK}:\Gal_K \to \GL_2(K_{\idK})\}$ indexed by
prime ideals $\idK$ of $K$. By \cite[Theorem 3.13]{BCDF1}, its
Nebentypus is trivial (i.e.\ the representation $\rho_{J_r^\pm (t_0),\idK}$
has determinant the cyclotomic character restricted to $\Gal_K$). Let
$N_r^{\pm}(t_0)$ denote the conductor of the family.
Then (see e.g.\ \cite[Proposition A.12]{CelineClusters}),
\begin{equation}
  \label{eq:gl2-conductor-relation}
\left(\frac{r-1}{2}\right)  v_{\id{r}}(N_r^{\pm}(t_0)) = v_{\id{r}}(\cond(J_r^\pm(t_0))).
\end{equation}
Keeping the previous notation, if $J^\pm_r(t_0)$ has potential good
reduction, we denote by $e$ the degree of the minimal extension of
$\Q_r^{\text{ur}}$ where it acquires good reduction.
%
%
\begin{prop}
  \label{prop:conductor-r}
  Keeping the previous notation, assume that $p,q>3$. Then
  Table~\ref{table:Cr-local} contains the valuation of
  $N_r^{\pm}(t_0)$ at the prime ideal $\id{r}$
  together with its local type.
  \begin{table}[H]
    \begin{tabular}{|c|c|c|c|c|c|c|c|}
      \hline
      $F(x)$ reducible in $\Q_r$ &  $r\mod 4$ &$ r\mid a$  & $ r\mid b$  &$v_\id{r}(N_r^-(t_0))$  & $v_\id{r}(N_r^+(t_0))$  &  $e$ & Type \\
      \hline
      \hline
      -- &--& \checkmark &$\times$ & $2$ & $1$ & --& special\\
      \hline
      --&--& $\times$ &$\checkmark$ & $ 2$ & $2$ & --& special\\
      \hline
      \checkmark&$1$ & $\times$& $\times$  &$2$ &$2$& $4$ & principal series\\
      \hline
      \checkmark&$3$ & $\times$& $\times$  &$2$ &$2$& $4$ &supercuspidal\\
      \hline
      $\times$&-- & $\times$& $\times$  &$3$ &$3$& $4r$ & supercuspidal \\

      \hline
    \end{tabular}
    \caption{Local information of $J_r^\pm (t_0)$ at $\id{r}$.}
    \label{table:Cr-local}
    \end{table}
\end{prop}
\begin{proof}
  Let $\J_r^\pm(a,b,c)$ denote the Jacobian of $C_r^\pm(a,b,c)$
  respectively.  By Lemma \ref{lemma:iso-curves} the abelian variety
  $\J_r^+(a,b,c)$ is isomorphic to $J_r^+(t_0)$, while $\J_r^-(a,b,c)$
  is isomorphic to $J^-(t_0) \otimes \theta_c$, where $\theta_c$ is the
  quadratic character of the extension $\Q(\sqrt{c})/\Q$ (and by abuse
  of notation its restriction to $\Gal_K$). We claim that it is enough
  to prove the stated results for the varieties
  $\J_r^{\pm}(a,b,c)$. In the $+$ case, the claim is clear, while for
  the $-$ curve, note that the character $\theta_c$ is ramified at a
  prime ideal $\id{p}$ if and only if $\id{p} \mid c$, in which case
  $\id{p} \nmid ab$ (by the primitive assumption). But the values of the
  last four rows of the table are invariant under quadratic twist.

\vspace{5pt}
  (1)  The curve $C_r^-(a,b,c)$: it is a particular instance of the
  family of hyperelliptic curves $C(z,s)$ studied in \cite{PedroLucas}
  (see equation (3) in loc.\ cit.) specialized at $z=c^2$ and
  $s=2c^r+4a^q$. Its cluster picture (over $K_\id{r}$) is given in
  \cite[Corollary 3.5]{PedroLucas}. In the notation of
  \cite{PedroLucas},
  \begin{equation}
    \label{eq:delta}
  \Delta=s^2-4z^r=16a^qb^p,  
  \end{equation}
  so $r\mid \Delta$ if
  and only if $r\mid ab$, and in that case $v_r(\Delta)\ge 3$ (by the
  assumption $p, q>2$).

  If $r\mid ab$, $\J_r^-(a,b,c)$ has potential multiplicative reduction
  at $\id{r}$, by \cite[Theorem 1.9]{MR4566695}. 
  The stated conductor
  exponent formula follows from \cite[Theorem
  5.4]{PedroLucas}. Moreover, $C_r^-(a,b,c)/K_\id{r}$ is a quadratic
  ramified twist of a semistable curve, by \cite[Corollary
  5.5]{PedroLucas}.
	
  Assume then that $r\nmid ab$. By \cite[Theorem 1.9]{MR4566695},
  $\J_r^-(a,b,c)$ has potentially good reduction at $\id{r}$. Its
  conductor exponent is computed in \cite[Theorem 5.4]{PedroLucas}, so
  it remains to compute the last two columns of the table.  Let $K'$
  be the splitting field of $F(x)$ (described in \cite[Proposition
  4.5]{PedroLucas}). By \cite[Theorem
  4.11]{PedroLucas},  the
  ramification index of $K'_\id{r}/K_\id{r}$ equals 1 if $F(x)$ is
  reducible over $\Q_r$ and $r$ otherwise. 
  
  Let $M/K'_\id{r}$ be any finite extension with ramification index
  equal to 4 (so the ramification index of $M/K_{\id{r}}$ coincides
  with the values of $e$ given in Table~\ref{table:Cr-local}). Then
  \cite[Theorem 1.9]{MR4566695}
   implies that $\J_r^-(a,b,c)$ has good
  reduction over $M$, and that there is no extension of $K_\id{r}$
  with ramification index $\le e$ where $\J_r^-(a,b,c)$
  obtains good reduction.
%
  This information is enough to compute the local type. If $F(x)$ is
  reducible, $e=4$, so the local type is either principal series or
  supercuspidal (induced by a character of order $4$ of the unramified
  quadratic extension of $K_{\id{r}}$). The first case occurs
  precisely when there is a totally ramified degree $4$ abelian
  extension of $K_{\id{r}}$, which by local class field theory, is
  equivalent to the condition $4 \mid \#\F^\times_\id{r}=r-1$. If
  $F(x)$ is irreducible, the conductor exponent is $3$, so the local
  type is that of a supercuspidal representation induced from a
  ramified quadratic extension of $K_{\id{r}}$ (see \cite[Corollary
  3.1]{MR3056552}).
	
%

\vspace{5pt}
(2) The curve $C_r^+(a,b,c)$: by \cite[Remark 2.10]{CelineClusters},
  the Swan part of $C_r^+(a,b,c)$ equals that of $C_r^-(a,b,c)$ (this
  is a particular instance of congruences of hypergeometric motives,
  since by \cite[Theorem 10.3]{GPV} the motives
  $\HGM((\frac{1}{r},-\frac{1}{r}),(1,1)|t_0)$ and
  $\HGM((\frac{2+r}{2r},-\frac{2+r}{2r}),(1,1)|t_0)$ are congruent
  modulo $2$, and the latter is a Galois conjugate of the hypergeometric motive
  $\HGM((\frac{1}{2r},-\frac{1}{2r}),(1,1)|t_0)$). Then, by
  \cite[Equation (39)]{PedroLucas} and (\ref{eq:gl2-conductor-relation}), the Swan component equals
  \[
    \mathfrak{n}_{\id{r}}^{\text{wild}}(\rho_{\J_r^+,\idK})=
    \begin{cases}
      1 & \text{if } r\nmid ab \text{ and } F(x) \text{ is irreducible},\\
      0 & \text{otherwise,} 
    \end{cases}
  \]
  where $\{\rho_{\J_r^+,\idK}\}$ is the compatible system of
  2-dimensional Galois representations attached to $\J_r^+(a,b,c)$.
  To compute the the tame conductor, note that the curve
  $C_r^+(a,b,c)$ is a particular instance of the family of
  hyperelliptic curves $C_r^+(s)$ defined in \cite[Definition
  3.2]{Martin}, whose cluster picture is given in \cite[Theorem
  4.21]{Martin}.  
  Then
  \cite[Theorem 1.9]{MR4566695} implies   that 
  $\J_r^+(a,b,c)$ has multiplicative reduction at $\id{r}$ if $r\mid a$
  and additive reduction otherwise, finishing the conductor computation.

  The proof of the stated results on the last two columns of the table
  for $\J_r^+(a,b,c)$ follows from similar arguments as the ones used
  in the previous case.
\end{proof}

Let $\id{r}$ be a prime ideal of $K=\Q(\zeta_q)^+\cdot \Q(\zeta_r)^+$
dividing $r$ and let $v_\id{r}(\mot^-)$ denote the $\id{r}$-th
valuation of the motive's conductor.
\begin{coro}
  \label{coro:conductor-at-r}
   Assume that  $p,q>2$. Then,
  the values of the conductor exponent of $\mot^\pm$ at $\id{r}$ are
  given in Table~\ref{table:Conductorr}. Moreover, if $q\nmid \frac{r-1}{2}$, the case 0 can be rule out in the first row.
  \begin{table}[H]
    \begin{tabular}{|c|c|c|c|c|}
      \hline
      $F(x)$ reducible in $\Q_r$  &$ r\mid a$  & $ r\mid b$  &$v_\id{r}(\mot^-)$  & $v_\id{r}(\mot^+)$ \\
      \hline
      \hline
      -- & \checkmark &$\times$ & $0, 1, 2$ & $0, 1, 2$ \\
      \hline
      -- & $\times$ &$\checkmark$ & $0, 1, 2$ & $0, 1, 2$ \\
      \hline
      \checkmark& $\times$& $\times$  &$2$ &$2$\\
      \hline
      $\times$& $\times$& $\times$  &$3$ &$3$\\
      \hline
    \end{tabular}
      \caption{Conductor exponent values of $\mot^{\pm}$ at $\id{r}$.}
      \label{table:Conductorr}
\end{table}	
\end{coro}

\begin{proof}
 
  Recall that $J_r^\pm(t_0)$ is a rational variety that become of $\GL_2$-type over $K$, and that $\mot^\pm$  is defined over $K$.  Note  that by
  Propositions~\ref{prop:darmon-hyper-relation} and \ref{prop:congruences}  we have a congruence over $F=\Q(\zeta_{N^\pm})$
  between $\mot^\pm$ and $J_r^\pm(t_0)$. Moreover, by Remark \ref{remark:twist}, they are congruent over $K$ up to a
  quadratic twist.  
  Hence,   the proof mimics that of Corollary~\ref{coro:conductor-at-3},
  replacing the results of Proposition~\ref{prop:Et-reduction} by
  those of Proposition \ref{prop:conductor-r}. 
  
  To prove the last statement, it is sufficient to prove that the
  residual representation $\bar{\rho}_{\J_r^\pm , \id{q}}$ is ramified
  at $\id{r}$ if $q\nmid\frac{r-1}{2}$ and $r\mid a$, where
  $\J_r^\pm=\Jac(C_r^\pm(a,b,c))$.  This follows from the theory of
  Mumford curves (see e.g.\ \cite[Theorem 7.6]{BCDF1}), since
  \[
    v_\id{r}(\Disc(C^\pm(a,b,c)))\equiv rv_\id{r}(r)\equiv \frac{r(r-1)}{2} \not\equiv 0 \pmod q,
  \]
  where the first congruence follows from \cite[Lemma 3.1]{PedroLucas}
  in the $-$ case, and from \cite[Corollary 4.12]{Martin} in general.
\end{proof}

\begin{remark}
  \label{rem:twist-control}
  As in Remark~\ref{remark:non-twist}, if
  Proposition~\ref{prop:congruences} holds over $K$ (or up to a
  quadratic unramified twist), then the exponent equals $2$ for the
  first two rows (except the case $r\mid b$ for the motive $\mot^+$)
  because a special local type of conductor valuation $2$ and trivial
  Nebentypus has residual representation of conductor valuation $2$.
\end{remark}

\subsubsection{Irreducibility of $F(x)$}
Let $F(x)$ be as in (\ref{eq:Def-F}). According to Corollaries
\ref{coro:conductor-at-3} and \ref{coro:conductor-at-r} the conductor
exponent of $J_r^{\pm}(t_0)$ sometime depends  on the irreducibility
of $F(x)$.  The following result, inspired by \cite[Proposition
7.23]{ChenKoutsianas}, provides a criterion in this direction. 

\begin{lemma}\label{lemma:irreducibility}
  For $(a,b,c)$ a solution to (\ref{eq:GFE}), let $d(a,c)$ be the
  constant term of $g(x)=F(x-2c^r-4a^q)$. Suppose $r\nmid ab$.  Then, $v_r(d(a,c))\ge1$ and
  \begin{enumerate}
  \item If $v_r(d(a,c))\ge 2$, $F(x)$ is reducible over $\Q_r$.
  \item If $v_r(d(a,c))=1$, $F(x)$ is irreducible over $\Q_r$.
  \end{enumerate}
Furthermore, the case depends only on the congruence of $(a,b,c)$ modulo $r^2$.
\end{lemma}

\begin{proof}
  By \cite[Lemma 2.7]{BCDF1},
  \[
    F(x)=\sum_{k=0}^{\frac{r-1}{2}}(-1)^kc_kc^{2k}x^{r-2k}+2c^r+4a^q, \quad
    \text{ where } c_k=\frac{r}{r-k}\binom{r-k}{k}.
  \]
  Since $r\mid c_k$ for all $0<k$, Fermat's Little Theorem implies
  that $r \mid d(a,c)$. Clearly $F(x)$ is irreducible if and only if
  $g(x)$ is irreducible.   When $v_r(d(a,c))= 1$
  the polynomial $g(x)$ is irreducible over $\Q_r$ by Eisenstein's
  criterion.
  
  Let us assume that $v_r(d(a,c))=2$, and consider the previous curve
  \[
    C:=C_r^-(a,b,c):y^2=F(x).
  \]
  By \cite[Lemma 3.1]{PedroLucas}, its discriminant equals
  $\Delta_C =
  (-1)^{\frac{r-1}{2}}2^{2(r-1)}r^r\Delta^{\frac{r-1}{2}}$, where
  $\Delta=16a^qb^p$ by (\ref{eq:delta}).
  
  Let $\id{r}$ be a prime in $K$ above $r$, and $M/K_\id{r}$ be an
  extension with ramification index equal to 4. We claim that $C/M$
  has good reduction. Then Table \ref{table:Cr-local} implies that
  $F(x)$ is reducible in $\Q_r$.
  
  Let $\pi$ be a uniformizer in $M$ and $v$ its normalized valuation
  so that $v(\pi)=1$. Then $v(r)=2(r-1)$. Since $r\nmid ab$, we have
  $v(\Delta_C)=v(r^r)=2r(r-1)$. The change of variable
  $x\to x-2c^r-4a^q$ leads to the model
  \[y^2=g(x)=x^r + \sum_{j=1}^{r-1}b_jx^j + d(a,c),\] where
  $r\mid b_j$ for all $1\le j \le r-1$. Now, applying the change of
  variables $x\to \pi^2x$, $y\to \pi^ry$ we obtain the model
  \[C' : y^2 = x^r  + \sum_{j=1}^{r-1}\frac{b_j}{\pi^{2(r-j)}}x^j + \frac{d(a,c)}{\pi^{2r}}.\]
  
  The assumption $v_r(d(a,c))\ge 2$ implies that $C'$ is integral.  By
  \cite[\S 1]{Lockhart}, $\Delta_{C'}=\Delta_C/\pi^{2r(r-1)}$, which
  is a unit so $C'$ is a model with good reduction over $M$.
  
   It follows from the proof that the result depends only on
  the congruence of the parameters $(a,b,c)$ modulo $r^2$.
%
\end{proof}

\subsection{Conductor at primes dividing $q$} This case can be deduced
from the previous one using the following two properties.
\begin{itemize}
\item For $a,b,c,d\in\Q$,
  \begin{equation}
    \label{eq:ismorphism-inv}
  \HGM((a,b),(c,d)|t) \simeq \HGM((-c,-d),(-a,-b)|t^{-1}),  
  \end{equation}
  as proven in \cite[Proposition 7.28]{GPV}.
\item For $a,b\in\Q$, the motive
  $\HGM((a+\frac{1}{2},-a+\frac{1}{2}),(b+\frac{1}{2},-b+\frac{1}{2})|t)$
  is isomorphic to the quadratic twist by $\sqrt{t}$ of the motive
  $\HGM((a,-a),(b,-b)|t)$ (see
  \cite[Proposition 2.11]{GP}).
\end{itemize}

Equation (\ref{eq:ismorphism-inv}) implies that
\[
\mot^+\simeq \HGM\left(\left(\frac{1}{q},-\frac{1}{q}\right),\left(\frac{1}{r},-\frac{1}{r}\right)\Big | -\frac{c^r}{a^q}\right).
\]
Similarly,
\begin{align*}
\mot^-\simeq & \ \HGM\left(\left(\frac{1}{q},-\frac{1}{q}\right),\left(\frac{1}{2r},-\frac{1}{2r}\right)\Big |-\frac{c^r}{a^q}\right)\\
\simeq& \ \HGM\left(\left(\frac{2+q}{2q},-\frac{(2+q)}{2q}\right)\left(\frac{(r+1)}{2r},-\frac{(r+1)}{2r}\right)\Big | -\frac{c^r}{a^q}\right)\otimes \theta_{-ac}.
\end{align*}
The latter is Galois conjugate to
\[
\HGM\left(\left(\frac{1}{2q},-\frac{1}{2q}\right)\left(\frac{1}{r},-\frac{1}{r}\right)\Big | -\frac{c^r}{a^q}\right)\otimes \theta_{-ac}.
\]
Therefore, the conductor exponent of
$\HGM^{\pm}(-a^q/c^r)$ at $\id{q}$ can be obtained from
Corollaries~\ref{coro:conductor-at-3} and~\ref{coro:conductor-at-r} by
exchanging $a \leftrightarrow c$ and $q \leftrightarrow r$ (noting
that the quadratic twist does not affect the conductor exponents),
proving the following result.
\begin{prop}
\label{prop:conductor-at-q}
Assume that $p,r>2$. Let $\id{q}$ be a prime ideal of $K=\Q(\zeta_q)^+\cdot \Q(\zeta_r)^+$ dividing
$q$. Then, the values of the conductor exponent of $\mot^\pm$ at
$\id{q}$ are given in Table~\ref{table:Conductorq}. Moreover, if $r\nmid \frac{q-1}{2}$, the case 0 can be rule out in the first row.
\begin{table}[H]
  \begin{tabular}{|c|c|c|c|c|}
    \hline
    $F(x)$ reducible in $\Q_q$  &$ q\mid c$  & $ q\mid b$  &$v_\id{q}(\mot^-)$  & $v_\id{q}(\mot^+)$ \\
    \hline
    \hline
    -- & \checkmark &$\times$ & $0, 1, 2$ & $0, 1, 2$ \\
    \hline
    -- & $\times$ &$\checkmark$ & $0, 1, 2$ & $0, 1, 2$ \\
    \hline
    \checkmark& $\times$& $\times$  &$2$ &$2$\\
    \hline
      $\times$& $\times$& $\times$  &$3$ &$3$\\
      \hline
    \end{tabular}
    \caption{Conductor exponent values of $\mot^{\pm}$ at $\id{q}$.}
    \label{table:Conductorq}
\end{table}
\end{prop}
%
\subsection{Conductor at primes dividing $2$}
Let $\id{q}_2$ be a prime in $K=\Q(\zeta_q)^+\cdot \Q(\zeta_r)^+$ above $2$. The conductor at
$\id{q}_2$ of $\mot^{\pm}$ is computed in \cite{GP}. The prime
$\id{q}_2$ is a tame prime for $\mot^+$, so its conductor and local
type information follows from Theorem 2.9 in loc.\ cit. The conductor
valuation of $\mot^-$ is the following.
\begin{prop}\label{prop:cond-at-2} Let $(a,b,c)$ a non-trivial
  solution to (\ref{eq:GFE}), where $q\ge 5$ and $r\ge3$. The
  conductor exponent of $\mot^-$ at $\id{q}_2$ is one of the
  following
\[
  \begin{cases}
    0,1,2  & \text{if } 2\mid a,\\
    5 & \text{if } 2\mid b \text{ and } a^q\equiv3c^r\pmod 4,\\
    6 & \text{if } 2\mid c \text{ and } 2\nmid v_2(c),\\
    4 & \text{if } 2\mid c, \ 2\mid v_2(c)  \text{ and } a\equiv3\pmod4,\\
    0, 1, 2 & \text{if } 2\mid c, \ 2\mid v_2(c)  \text{ and }  a\equiv1\pmod4.\\
  \end{cases}
\]
\end{prop}
\begin{proof}
See \cite[Theorem 7.2]{GP} for $t_0=-a^q/c^r$.
\end{proof}
\section{Motives coming from trivial solutions}
\label{section:trivial}
Recall that a solution to (\ref{eq:GFE}) is called trivial if one of
its coordinates is zero.  There are six different trivial primitive solutions,
namely
\[
  (0, \pm 1, \mp1), \ (\pm1, 0, \mp1), \ (\pm1, \mp1, 0).
\]
They correspond to the specialization $t_0=0, 1, \infty$,
respectively. Given $a,b,c,d\in\Q$, the specialization of the motive
$\HGM((a,b),(c,d)|t)$ at any of these three points yields a ``singular
motive''. If the monodromy matrix at the point $t_0$ has infinite order, then
the singular motive does not correspond to any automorphic form (as
occurs in Fermat's Last Theorem, where a trivial solution corresponds
to a singular elliptic curve with ``multiplicative
reduction''). However, when the monodromy matrix has finite order, it
does correspond to a motive with complex multiplication, as we now
prove. First, let us recall the following necessary definition.

\begin{defi}
	The parameters $(a, b),(c, d)$ are called \textit{generic} if no element of the set $\{a - c, a -
		 d, b - c, b - d\}$ is an integer.
\end{defi}

\begin{thm}
  Let $(a,b),(c,d)$ be generic parameters such that $a+b$ and
  $c+d$ are integers, but $c-d$ is not. Let $N$ be their least
  common denominator. Then $\HGM((a,b),(c,d)|0)$ is a motive with
  complex multiplication unramified outside primes dividing $N$. 
  
  If
  $\idF$ is a prime ideal of $F$ not dividing $N$, then the trace of
  $\Frob_{\idF}$ acting on $\HGM((a,b),(c,d)|0)$ equals
  \begin{equation}
    \label{eq:t=0}
    -\gent(-1)^{(d-b)(\norm(\idF)-1)}\norm(\idF)^\varepsilon
    \left(J(\gent^{(d-b)N},\gent^{(b-c)N})+\gent(-1)^{(b-c)N}J(\gent^{(d-c)N},\gent^{(a-d)N})\right),    
  \end{equation}
  where
  \[
    \varepsilon =
    \begin{cases}
      -1 & \text{ if }a,c \not \in \ZZ,\\
      0 & \text{ otherwise},
    \end{cases}
  \]
  $\chi_\idF$ is as in (\ref{eq:char-def}) and $J(\cdot,\cdot)$ denotes the usual Jacobi sum.
\label{thm:trivial-sols}
\end{thm}

\begin{proof}
  The proof mimics that of   \cite[Theorem \ref{thmA}.1]{GPV}. Let $A = (d-b)N$,
  $B=(b-c)N$, $C=(a-d)N$ and $D=dN$. Euler's curve is defined by the equation
  \[
\bfC:y^N=x^A(1-x)^B(1-tx)^Ct^D.    
\]
Write $t=z^N$, so the curve can be written in the form
  \begin{equation}
    \label{eq:equation-C}
\bfC:y^N=x^A(1-x)^B(1-z^Nx)^C.    
\end{equation}
The model reduces modulo $z$ to the curve
  \[
  \bfC_1: y^N = x^A(1-x)^B.
\]
Since the parameters are generic, $N \nmid A$ and $N\nmid B$. The
hypothesis on $c-d$ implies that $N \nmid A+B=N(d-c)$, hence
$\dim(\Jac(\bfC_1)^{\text{new}})=\frac{\varphi(N)}{2}>0$ (see formula (16) of \cite{MR2005278}). The curve
$\bfC_1$ is a component of the stable model of $\bfC$.

Set $x_1=z^Nx$, and $y_1=z^{A+B}y$. Then $\bfC$ transforms into
\[
y_1^N=x_1^A(z^N-x_1)^B(1-x_1)^C.
\]
Its reduction modulo $z$ equals
\[
\bfC_2:y_1^N=(-1)^Bx_1^{A+B}(1-x_1)^C.
\]
This is another component of genus $\frac{\varphi(N)}{2}$. Both curves
have complex multiplication, and their number of points can be
computed using   \cite[Theorem 7.23]{GPV}.
\end{proof}
\begin{remark}
  Using (\ref{eq:ismorphism-inv}) we get a similar result for the specialization at $t_0=\infty$.
\end{remark}
Applying the result to the motive $\mot^{\pm}$ we get that its
specialization at the point $t_0=1$ (corresponding to the trivial
solution $(\pm 1, 0, \mp 1)$) will not correspond to any Hilbert
modular form, but its specialization at the point $t_0=0$ and
$t_0=\infty$ does.  Let us present a good estimate for the conductor of
the attached Galois representation. 
\begin{lemma}
  \label{lemma:cond-at-t=0}
  Let $M$ denote the conductor of the specialization of $\HGM^\pm$ at
  $t_0=0$. Let $\varepsilon_2$ (respectively $\varepsilon_r$,
  $\varepsilon_q$) denote the valuation of $M$ at a prime ideal of
  $K$ dividing $2$ (respectively dividing $r$ and $q$). Then
  \begin{align*}
    \varepsilon_2=\begin{cases}
                    0 & \text{for } \  \HGM^+,\\
                    0,1,2 & \text{for } \ \HGM^-,
                  \end{cases}
    \quad
    \varepsilon_r\in\{0,1,2\},
    \quad
    \varepsilon_q=\begin{cases}
                    2& \text{ if } v_q(4^{q-1}-1)\ge 2,\\
                    3 & \text{ if } v_q(4^{q-1}-1)=1.
 		\end{cases}
  \end{align*}
\end{lemma}
 \begin{proof}
  In \S \ref{sec:conductor} we have studied the conductor of
  $\HGM^\pm$ at values $t=-a^q/c^r$, where $(a,b,c)$ is a solution to
  (\ref{eq:GFE}). The value $t_0=0$ corresponds to
  $(a,b,c)=(0,\pm 1,\mp 1)$. Since in this case $2r\mid a$, the values
  of $\varepsilon_2$ and $\varepsilon_r$ follow from Proposition
  \ref{prop:cond-at-2} and Corollaries \ref{coro:conductor-at-3} and
  \ref{coro:conductor-at-r}, respectively. To compute the value of
  $\varepsilon_q$ we may use, as in the proof of
  Proposition~\ref{prop:conductor-at-q}, that $(c,b,a)=(\mp1,\pm1,0)$
  is a solution to $z^r+y^p+x^q=0$; that is,  we must interchange
  $a \leftrightarrow c$ and $q \leftrightarrow r$ in
  Corollaries~\ref{coro:conductor-at-3}
  and~\ref{coro:conductor-at-r}. This implies
  $\varepsilon_q\in\{2,3\}$ and its value depends on the
  irreducibility of
  \begin{align}\label{eq:Def-F-invertida}
  	F(x)&=a^qf\Big(\frac{x}{a}\Big)+2a^q + 4c^r\\
  	&=x^q \mp 4,\notag
  \end{align}
  since $a=0$
  and $c=\mp1$. Then, the result follows by noting that $F(x)$ is
  irreducible over $\Q_q$ if and only if $v_q(4^{q-1}-1)=1$, by Lemma
  \ref{lemma:irreducibility}.
\end{proof}
Similarly, for $t_0=\infty$ the following holds.
\begin{lemma} \label{lemma:cond-at-t=inf} Let $M$ denote the conductor
  of the specialization of $\HGM^\pm$ at $t_0=\infty$. Let
  $\varepsilon_2$ (respectively $\varepsilon_r$, $\varepsilon_q$)
  denote the valuation of $M$ at a prime ideal of $K$ dividing $2$
  (respectively dividing $r$ and $q$). Then
  \begin{align*}
    \varepsilon_2=\begin{cases}
                    0 & \text{for } \ \HGM^+,\\
                    0,1,2 & \text{for } \ \HGM^-,
                  \end{cases}
    \quad
    \varepsilon_q\in\{0,1,2\},
    \quad
    \varepsilon_r=\begin{cases}
                    2& \text{ if } v_r(4^{r-1}-1)\ge 2,\\
                    3 & \text{ if } v_r(4^{r-1}-1)=1.
 		\end{cases}
  \end{align*}
\end{lemma}
 
\begin{proof}
  Follows from Lemma~\ref{lemma:cond-at-t=0}, applying (\ref{eq:ismorphism-inv}).
\end{proof}

\begin{remark}\label{rem:cond-trivial}
	As in some of the results in the previous section, the value of $\varepsilon_r$ in Lemma \ref{lemma:cond-at-t=0} could be more restrictive under some conditions. For example, if $q\nmid \frac{r-1}{2}$, or if Proposition \ref{prop:congruences} holds over $K$ (see Remark \ref{rem:twist-control}). The same applies to the values of $\varepsilon_q$ in Lemma \ref{lemma:cond-at-t=inf}.
	
\end{remark}

\section{An obstruction to the method: ghost solutions}
\label{section:obstruction}
The existence of a non-trivial primitive solution to (\ref{eq:GFE}) 
typically prevents a successful application of the modular method. 
As briefly explained in the introduction, there is another obstruction
which, to our knowledge, has not previously been observed in the literature,
is inherent to Darmon's program, and which we refer to as \emph{ghost solutions}.

As clearly explained in \cite{Darmon}, Darmon's program consists of constructing, from a hypothetical solution
to the generalized Fermat equation, a number field
 (obtained as the fixed field of a Frey
representation) with a small ramification set. A ghost solution is a
rational parameter $t_0 \neq 0$ for which the conductor of the specialized
motive $\HGM^{\pm}(t_0)$ has the same small ramification set.
To illustrate the phenomenon, let us consider Fermat's equation
with signature $(p,p,p)$.
 To a putative solution $(a,b,c)$ one attaches the
Frey elliptic curve
\begin{equation}
  \label{eq:Legendre}
  E_t:y^2=x(x-1)(1-tx),
\end{equation}
specialized at $t_0=-\frac{a^p}{c^p}$. Denote by $\rho_{t_0,p}$ the
Galois representation attached to $E_{t_0}$
twisted by the quadratic character $\theta_c$. Then the residual representation
$\bar{\rho}_{t_0,p}$ is unramified outside $2$ (as expected) and
is the reduction of the Galois representation associated to a modular form whose level is only supported at
$2$.

For $p$ an odd prime, consider the  Diophantine equation
\[
x^p + 2 + z^p=0.
\]
Its set of solutions consists of the point $(-1,-1)$. Specializing
(\ref{eq:Legendre}) at $t_0=-\frac{-1}{-1}=-1$ we obtain the elliptic
curve \lmfdbec{32}{a}{3}, a curve with good reduction outside
$2$. Moreover, it has conductor $N_1=32$. Similarly, if we consider the equations
\[
2 + y^p + z^p=0, \qquad x^p + y^p +2=0,
\]
the point $(-1,-1)$ is always a solution, and corresponds
(respectively) to the specializations $t_0=2$ (a curve isomorphic to
\lmfdbec{32}{a}{3}) and $t_0=\frac{1}{2}$ (corresponding to the
elliptic curve \lmfdbec{64}{a}{3}). Their conductors are $N_2=32$ and
$N_3=64$, respectively.  The three specializations obtained
do not come from real solutions of Fermat's equation, but from a
variant of it (which motivates the terminology \emph{ghost solution}).

It is easy to verify, using Tate's algorithm, that if $(a,b,c)$ is a
solution to Fermat's equation of signature $(p,p,p)$ and
$t_0=-\frac{a^p}{c^p}$, then the conductor of the specialized Frey's
elliptic curve $E_{t_0}$ might have valuation $5$ or $6$ at $2$. There
is no clear reason why the elliptic curves \lmfdbec{32}{a}{3} and
\lmfdbec{64}{a}{3} (coming from ghost solutions) should not be related to genuine solutions, so the modular method ``theoretically'' fails to prove
Fermat's Last Theorem.  However, if we impose the extra hypothesis on
$b$ being even, then $N=2$ and the other spaces of modular forms do
not need to be considered. It is a remarkable coincidence that we can
impose a parity condition on $b$, as this is only true when the
original equation has symmetries, which is not the case for an
arbitrary generalized Fermat equation.

A similar phenomenon occurs while studying the generalized Fermat
equation (\ref{eq:GFE}).
\begin{defi}\label{def:ghost}
  A \emph{ghost solution} for the exponents $(q,r)$ is a tuple
  $(a,c,s,\ell,m,n,u,v) \in \Z^2 \times \Z_{\ge 0}^6$, with
  $ac \neq 0$, satisfying
\begin{equation}\tag{\ref{eq:ghost-eqn}}
q^s r^\ell a^q \pm q^m r^n+ q^ur^vc^r=0.
\end{equation}
\end{defi}
If $(a,c,s,\ell,m,n,u,v)$ is a ghost solution and we set
$t_0:=-\frac{a^qq^{s-u}r^{\ell-v}}{c^r}$, then Theorem
\ref{thm:mot-+-prop} implies that the motive $\mot^+$ is unramified
outside prime ideals of $K$ dividing $qr$.

\begin{exm}
  If $q >3$ is odd, the tuple $(-1,-2,0,0,0,2,0,0)$ is a ghost solution
  for exponents $(q,3)$, since
\[
(-1)^q +  3^2+  (-2)^3=0.
\]
It corresponds to the specialization $t_0=-\frac{1}{8}$. The tuple $(1,-2,0,2,0,0,0,0)$, satisfying
\[
3^2 + (-1)^q +  (-2)^3=0,
\]
corresponds to the ghost solution for exponents $(q,3)$ with
specialization $\frac{9}{8}$.

Similarly, if $q>2$ is odd, $(-1,3,0,0,0,3,0,0)$ is a ghost solution
for exponents $(q,2)$, since
\[
(-1)^q +(-2)^3  +  3^2=0.
\]
It corresponds to the specialization $t_0=\frac{1}{9}$. Similarly, the
tuple $(-1,3,0,3,0,0,0,0)$ corresponding to the equation
\[
2^3(-1)^3 - 1   +  3^2=0
\]
gives the ghost solution with specialization $\frac{8}{9}$. These
tuples are \emph{Catalan's solutions}, since they are related to
solutions of Catalan's equation
\[
x^n - y^m=1,
\]
whose complete set of solutions is $\{(3^2,2^3)\}$ as proved by Mih\u{a}ilescu in
\cite{MR2076124}.
\end{exm}
The modular method has good chances to prove non-existence of
solutions precisely when the specialization of $\HGM^{+}(t)$ at any
ghost solution is a motive whose conductor exponents are not the ones
described in Tables~\ref{table:cond3}, \ref{table:Conductorr} and
\ref{table:Conductorq}. Otherwise, using current techniques, we can
only expect to obtain partial results.
\begin{remark}
  While working with the motive $\HGM^-(t)$ we need to also allow
  powers of $2$ in (\ref{eq:ghost-eqn}).
\end{remark}  

\subsection{Conductors of Catalan's ghost solutions}
Let $\id{q}$ be the prime ideal of $K=\Q(\zeta_q)^+$ dividing $q$ and let
$f(x)$ be the polynomial defined in (\ref{eq:f-pol}).  Throughout this section we denote
$\HGM(t) =
\HGM\left(\left(\frac{1}{4},-\frac{1}{4}\right),\left(\frac{1}{q},-\frac{1}{q}\right)|t\right)$
for the exponents $(q,2)$ and
$\HGM^+(t) =
\HGM\left(\left(\frac{1}{3},-\frac{1}{3}\right),\left(\frac{1}{q},-\frac{1}{q}\right)|t\right)$
for the exponents $(q,3)$.
\begin{lemma}
  \label{lemma:Catalan-2} If $q\ge5$, the specialization of
  $\HGM(t)$
  at $t_0=\frac{1}{9}$ has conductor $2^6\cdot\id{q}^{\varepsilon_q}$, with
\[\varepsilon_q=\begin{cases}
	2 & \text{ if } F(x) \text{ is reducible over } \Q_q,\\
	3  & \text{ if } F(x) \text{ is irreducible over } \Q_q,\\
\end{cases}\]
where $F(x)=f(x) + 34 $.
\end{lemma}
\begin{proof}
  The proof mimics that of Corollary~\ref{coro:conductor-at-r}. We know that $\HGM(t)$ is defined over the totally real field $K$.  By (\ref{eq:ismorphism-inv}) and  \cite[Theorem 9.3]{GP}, over $F=\Q(\zeta_{4q})$,  it is congruent modulo 2
  to
  \begin{align}\label{eq:Tilde-Mot}
	   \widetilde{\HGM}(t)=\HGM\left(\left(\frac{1}{q},-\frac{1}{q}\right),\left(1,1\right)\Big|\frac{1}{t}\right).
  \end{align}


Then, by Proposition \ref{prop:darmon-hyper-relation}, the conductor at $\id{q}$ of $ \widetilde{\HGM}(\frac{1}{9})$ equals that of the 2-dimensional Galois representation attached to $J^+_q(9)$. The latter follows from Proposition~\ref{prop:conductor-r}, noting that  $t_0$ corresponds to the solution $(a,b,c)=(-1,-2,3)$, and using (\ref{eq:Def-F-invertida}) for the formula of $F(x)$. It is  worth noting that Proposition~\ref{prop:conductor-r} can be applied, since the hypothesis $p,q>3$ is only used to compute the conductor in the two first rows, which is not our case.  

 By Remark~\ref{remark:twist},
  over $K$, the motive $\HGM(t)$ might be isomorphic not to
  $\widetilde{\HGM}(t)$, but to a ramified quadratic twist of it.  But
  the fact that $e=4$ or $4r$ in Proposition~\ref{prop:conductor-r}
  implies that conductor is invariant under ramified quadratic twists,
  because the local type is either supercuspidal, or a ramified
  principal series given by an order $4$ character and its inverse.
  

\vspace{5pt}

  Let $\id{n}$ be a prime ideal of $K$ dividing $2$.  To compute the
  conductor at $\id{n}$, we use its congruence (over $F$) modulo
  $\id{q}$ with the motive
  \[
    \HGM\left(\left(\frac{1}{4},-\frac{1}{4}\right),\left(1,1\right)|t\right).
  \]
  The latter motive corresponds (see \cite{Cohen} and also
  \cite[Section 3.3]{GP}) to the elliptic curve
  \begin{equation}\label{eq:ellcurve-2}
  	    E_2(t) : y^2+xy=x^3+\frac{t}{64}x.
  \end{equation}
  The conductor exponent of $E_2$ specialized at $t_0=\frac{1}{9}$ at
  $2$ is $6$, as can be easily computed by running Tate's algorithm. Furthermore, by \cite[Table
  1]{DFV}, its local type
  corresponds to the quadratic twist by $\theta_{8}$ of a
  supercuspidal representation obtained by inducing a character of
  order $4$ from a ramified quadratic extension. 
   Since $2$ is unramified in $K/\Q$, the
  twisted curve has the same local type at $\id{n}$ over $K$, and the
  local type is preserved by the congruence, because $q$ is odd. A
  similar argument proves that the result is also true over~$K$.
\end{proof}

\begin{lemma}
	\label{lemma:Catalan-2-8/9} If $q\ge5$, the specialization of
	$\HGM(t)$
	at $t_0=\frac{8}{9}$ has conductor $2^5\cdot\id{q}^{\varepsilon_q}$, with
	\[\varepsilon_q=\begin{cases}
		2 & \text{ if } F(x) \text{ is reducible over } \Q_q,\\
		3  & \text{ if } F(x) \text{ is irreducible over } \Q_q,\\
	\end{cases}\]
	where $F(x)=2f(x/2) - 5\cdot2^{q-1}$.
\end{lemma}
\begin{proof}
As in the proof of Lemma~\ref{lemma:Catalan-2}, we have that 
$\HGM(t)$ is congruent modulo $2$ to $\widetilde{\HGM}(t)$, defined in (\ref{eq:Tilde-Mot}). Hence, to determine the value of $\varepsilon_q$ we need to study the conductor of $C_q^+\!\left(\frac{9}{8}\right)$ at $\id{q}$. As in Lemma~\ref{lemma:iso-curves}, it is easy to see that 
$C^+\!\left(\frac{9}{8}\right)$ and $C^-\!\left(\frac{9}{8}\right)/\Q(\sqrt{2})$ are isomorphic to the hyperelliptic curves
\[
\bfC^+ :\; y^2 = (x+4)\bigl(2f(x/2) - 5\cdot 2^{\,q-1}\bigr), \qquad
\bfC^- :\; y^2 = 2f(x/2) - 5\cdot 2^{\,q-1},
\]
respectively. Thus, it suffices to study the conductor of $\bfC^+$ at $\id{q}$. This can be done identically as in the proof of Proposition~\ref{prop:conductor-r}, noting that $\bfC^-$ equals the curve $C(z,s)$ defined in \cite[Equation~(3)]{PedroLucas}, for $z = 2^2$ and $s = -5\cdot 2^{\,q-1}$. To finish the computation of $\varepsilon_q$, the same argument as in the proof of Lemma~\ref{lemma:Catalan-2} holds.

To compute the conductor exponent at some prime $\id{n}$ dividing $2$, we can again use that $\HGM(t)$ is congruent modulo $\id{q}$ to the motive corresponding to the elliptic curve $E_2(t)$, defined in (\ref{eq:ellcurve-2}). The conductor exponent at $2$ of $E_2$ specialized at $t_0=\frac{8}{9}$ is $5$ and, by \cite[Table~1]{DFV}, its local type is given by a supercuspidal representation obtained by inducing a character of order 4 from  a ramified quadratic extension. 
Then the result follows as in Lemma~\ref{lemma:Catalan-2}.
\end{proof}

\begin{lemma}
\label{lemma:Catalan-3}  
If $q\ge 5$, the specialization of
$\HGM^+(t)$
at $t_0=-\frac{1}{8}$ has conductor $3^3\cdot\id{q}^{\varepsilon_q}$,
with
\[\varepsilon_q=\begin{cases}
	2 & \text{ if } F(x) \text{ is reducible over } \Q_q,\\
	3  & \text{ if } F(x) \text{ is irreducible over } \Q_q,\\
\end{cases}\]
where $F(x)=f(x) - 34$.
\end{lemma}
\begin{proof} The argument to compute the value of $\varepsilon_q$ follows as in Lemma \ref{lemma:Catalan-2}. That is, we can apply Proposition~\ref{prop:conductor-r} via a congruence modulo $3$ (with 
	the same subtlety of working over $K$ instead of $F$).  Note that  $t_0$ corresponds to the solution $(a,b,c)=(-1,3,2)$, so the formula of $F(x)$ follows from (\ref{eq:Def-F-invertida}).
    In this case, the formula for $F(x)$ follows from (\ref{eq:Def-F-invertida}), using that $(a,b,c)=(-1,3,2)$.
  
    Let $\id{n}$ be a prime ideal of $K$ dividing $3$. Note that
    Proposition \ref{prop:Et-reduction} does not apply because
    $v_3(1/t_0-1)=2$. By Proposition \ref{prop:congruences},
    $\HGM^+(t)$ is congruent modulo $\id{q}$ to the elliptic curve
    $E_3^+(t)$ defined in (\ref{eq:r=3-curves}).  Using Tate's
    algorithm, it is easy to see that the conductor of $E_3^+(-1/8)$
    at $3$ is $3$, 
    so its local type corresponds to a supercuspidal representation
    obtained from the induction of an order $6$ character of a
    ramified quadratic extension $\Q_3$ (see Table 1 of
    \cite{DFV}). Since $q>3$, $K/\Q$ is unramified at $3$, so
    $\HGM^+(-1/8)$ has the same local type and conductor exponent at
    $\id{n}$ as the elliptic curve.
\end{proof}

\begin{lemma}
	\label{lemma:Catalan-3-9/8}  
	Suppose $q\ge 5$, $q\notin \{19,37\}$. Then, the specialization of
	$\HGM^+(t)$
	at $t_0=\frac{9}{8}$ has conductor $3^3\cdot\id{q}^{\varepsilon_q}$,
	with 
	\[\varepsilon_q=\begin{cases}
		2 & \text{ if } F(x) \text{ is reducible over } \Q_q,\\
		3  & \text{ if } F(x) \text{ is irreducible over } \Q_q,\\
	\end{cases}\]
	where $F(x)=3f(x/3) - 14\cdot3^{q-2}$.
\end{lemma}

\begin{proof}
	As in the proof of Lemma~\ref{lemma:Catalan-2}, we have that 
	$\HGM^+(t)$ is congruent modulo $3$ to $\widetilde{\HGM}(t)$, defined in (\ref{eq:Tilde-Mot}). Hence, to determine the value of $\varepsilon_q$ we need to study the conductor of $C_q^+\!\left(\frac{8}{9}\right)$ at $\id{q}$. As in Lemma~\ref{lemma:iso-curves}, it is easy to see that 
	$C^+\!\left(\frac{8}{9}\right)$ and $C^-\!\left(\frac{8}{9}\right)/\Q(\sqrt{3})$ are isomorphic to the hyperelliptic curves
	\[
	\bfC^+ :\; y^2 = (x+6)\bigl(3f(x/3) -14\cdot3^{q-2}\bigr), \qquad
	\bfC^- :\; y^2 = 3f(x/3) -14\cdot3^{q-2},
	\]
	respectively. 
	Note that $\bfC^-$ equals the curve $C(z,s)$ defined in \cite[Equation~(3)]{PedroLucas}, for $z = 3^2$ and $s = -14\cdot3^{q-2}$. Following the notation in \cite{PedroLucas}, $\Delta$ (see equation (\ref{eq:delta})) is given by \[\Delta=z^q-4s^2=-3^{2q-4}\cdot19\cdot37.\]
	Therefore, if $q\neq19, 37$, we get that $q\nmid \Delta$. The rest of the proof to compute $\varepsilon_q$ follows as in the previous lemmas, which repeat the arguments in Proposition~\ref{prop:conductor-r}.
	
	The conductor exponent at a prime dividing 3 follows in the same way as Lemma \ref{lemma:Catalan-3}.
\end{proof}


\subsection{Specializations with Complex Multiplication}
We have seen in \S \ref{section:trivial} that the specializations
$t_0=0, \infty$ (arising from trivial solutions) correspond to abelian
varieties with complex multiplication (hence Hilbert modular forms
with complex multiplication). 

\begin{prop}
  \label{prop:non-cm}
  Let $q>r$ be odd primes. Let $(a,b,c)$ be a non-trivial solution to
  (\ref{eq:GFE}) with $p \ge 3$ and set $t_0=-\frac{a^q}{c^r}$. If the
  motive $\mot^{+}$ has complex multiplication, then there exists
  $m,n \in \Z_{\ge 0}$ such that $(a,c,0,0,mp,np,0,0)$ is
  a ghost solution for the exponents $(q,p,r)$.
\end{prop}

\begin{proof}
  Abelian varieties with complex multiplication have potentially good
  reduction at all primes, by \cite[Corollary 1]{MR236190}. By
  Theorem~\ref{thm:mot-+-prop}, all primes dividing $b$ and prime to
  $qr$ are primes of multiplicative reduction so $b$ is supported
  only at the primes $q$ and $r$. If $b=\pm q^m r^n$, then we have the
  relation
  \[
   a^q \pm q^{mp}r^{np} + c^r = 0,
 \]
 so $(a,c,0,0,mp,np,0,0)$ is a ghost solution.
%
%
%
%
%
\end{proof}
%
%
%
%
%
%
\begin{remark}
  In general we do not expect ghost solutions to have complex
  multiplication, but we do not know how to prove this fact. This
  problem is related to that of computing the set of values for which
  the specialized motive has complex multiplication (as explicitly
  done in \S 5 of \cite{MR971994}; unfortunately it depends on a
  result of \cite{MR931211} which is known to have a gap in its
  proof).
\end{remark}

\subsection{Computing some sets of ghost solutions}
It is not clear to us how, given the exponents $(q,r)$, one can compute all tuples $(a,c,s,\ell,m,n,u,v)$ satisfying~(\ref{eq:ghost-eqn}). What is
easy to do is to run some numerical experiments, searching for
solutions to (\ref{eq:ghost-eqn}) in boxes and compute for the
solutions found the conductor of the respective $2$-dimensional Galois
representation. When the conductor lies within the set of conductors
coming from hypothetical solutions to~(\ref{eq:GFE}) then the modular
method will fail to prove complete results on non-existence of
solutions. Otherwise, our computations (if complete) allow to
determine which exponents of the generalized Fermat equation should be
studied if we expect the modular method to work (always conditional on
the validity of Conjecture~\ref{conj:large-image} used to discard
trivial solutions). Assuming that our numerical experiments succeed
 in
finding all ghost solutions, we deduce that for any pair of exponents
$(q,r)$ with $3 < r < q \le 13$ either there are no ghost solutions,
or their conductor do not match the conductor of a real solution. We
alert the reader that the dimension of the space of Hilbert modular
forms for exponents different from the one studied in the present
article (namely $q=5$ and $r=3$) is too large, making the newform
decomposition infeasible with our electronic resources.

We present the ghost solutions obtained for different exponents
$(q,r)$ in the range $3 \le r < q \le 13$ except the case $(5,3)$ that
will be studied in detail later. For each value of $(q,r)$, let
$\id{q}$ (respectively $\id{r}$) be a prime ideal of $K$ dividing $q$
(respectively $r$). We include in the tables the valuation at $\id{q}$
and $\id{r}$ of the conductor of the attached $2$-dimensional
representation.

\subsubsection{The case $(q,r)=(7,3)$} We found the following solutions.
\begin{table}[H]
  \begin{tabular}{c|c|c|c|c|c|c|c}
  	
    $t_0$ & $-\frac{1}{8}$ & $\frac{1}{8}$ & $\frac{1}{64}$ & $\frac{7}{8}$ & $\frac{9}{8}$ & $\frac{63}{64}$ & $\frac{128}{125}$\\
    \hline
    $(v_{\id{q}}(N),v_{\id{r}}(N))$ & $(3,3)$ & $(3,6)$ & $(3,6)$ & $(3,9)$ & $(3,3)$ & $(3,9)$ & $(4,3)$
  \end{tabular}
  \end{table}
 Only Catalan's solutions appear in spaces of modular forms
 associated with solutions to~(\ref{eq:GFE}).

\subsubsection{The case $(q,r)=(7,5)$} We found the following two solutions.
\begin{table}[H]
  \begin{tabular}{c|c|c}
  	
    $t_0$ & $\frac{7}{32}$ & $\frac{25}{32}$\\
    \hline
    $(v_{\id{q}}(N),v_{\id{r}}(N))$ & $(3,9)$ & $(3,6)$
  \end{tabular}
  \end{table}
  None of them lie in a space coming from solutions to the generalized Fermat equation.

\subsubsection{The case $(q,r)=(11,3)$}  We found the following four solutions.
\begin{table}[H]
  \begin{tabular}{c|c|c|c|c}
    $t_0$ & $-\frac{3}{8}$ & $-\frac{1}{8}$ & $\frac{9}{8}$ & $\frac{11}{8}$\\
    \hline
    $(v_{\id{q}}(N),v_{\id{r}}(N))$ & $(4,8)$ & $(3,3)$ & $(3,3)$ & $(4,13)$
  \end{tabular}
  \end{table}
  The two middle values do lie in an allowed conductor, so the classical elimination step will fail to prove non-existence of solutions.
  \subsubsection{The case $(q,r)=(11,5)$} We found a unique solution
  corresponding to $t_0=\frac{2048}{2673}$. Its conductor valuation at
  a prime ideal dividing $5$ is $2$, while its conductor valuation at
  a prime dividing $11$ is $8$. Then it does not lie in any space
  related to solutions to (\ref{eq:GFE}).
\subsubsection{The case $(q,r)=(11,7)$}  We found the following two solutions.
\begin{table}[H]
  \begin{tabular}{c|c|c}
    $t_0$ & $\frac{7}{128}$ & $\frac{121}{128}$ \\
    \hline
    $(v_{\id{q}}(N),v_{\id{r}}(N))$ & $(6,3)$ & $(6,13)$ 
  \end{tabular}
  \end{table}
  Both conductors do not match the conductor of a newform coming from a solution to (\ref{eq:GFE}).
\subsubsection{The case $(q,r)=(13,3)$}  We found only the two solutions.
\begin{table}[H]
  \begin{tabular}{c|c|c}
    $t_0$ & $-\frac{1}{8}$ & $\frac{9}{8}$ \\
    \hline
    $(v_{\id{q}}(N),v_{\id{r}}(N))$ & $(2,3)$ & $(3,3)$ 
  \end{tabular}
  \end{table}
  They both have conductors matching conductors of real solutions to
  the generalized Fermat equation.  
  
  \vspace{7pt}
  
  Notice that when $r=3$ and
  $t_0=-\frac{1}{8}$ or $t_0=\frac{9}{8}$, the conductor exponent at $\id{q}$ and $\id{r}$
  matches the value computed in Lemma \ref{lemma:Catalan-3} and Lemma \ref{lemma:Catalan-3-9/8}, respectively. 
  
  \vspace{5pt}
  We did not find any solution for the exponents $(q,r)=(13,5), (13, 7)$ nor $(13,11)$.
  
%
%
\section{Some results on residual images}
%
%


Keeping the previous notation, let $\id{p}$ be a prime ideal of
$K=\Q(\zeta_q)^+\cdot\Q(\zeta_r)^+$, and let $\rho_{\id{p}}^{\pm}$ denote
the $\id{p}$-th Galois representation attached to the motive
$\mot^{\pm}$, as in (\ref{eq:gal-rep}).  As explained in the
introduction, the third step of the modular method (needed to apply a
level lowering result) depends on the residual representation
$\bar{\rho}^\pm_{\id{p}}$ being absolutely irreducible.  Since $K$ is
a totally real field, this is equivalent to being irreducible. In
this direction, let us recall the following result, which is
\cite[Corollary 1]{MR3690598}.

\begin{thm}\label{thm:irredasym}
  Let $K$ be a totally real field, $\id{q}$ a prime ideal of $K$ and
  $g$ a positive integer. Then there exists a constant $C(K,g,\id{q})$
  such that for all primes $p>C(K,g,\id{q})$ and all $g$-dimensional
  abelian varieties $A/K$ satisfying:
  \begin{enumerate}[(i)]
  \item $A$ has potentially good reduction at $\id{q}$,
  \item $A$ is semistable at the primes of $K$ dividing $p$,
  \item $A$ is of $\GL_2$-type with multiplications by some totally
    real field $F$,
  \item all endomorphisms of $A$ are defined over $K$ (i.e.\
    $\End_K(A) = \End_{\bar{K}}(A)$),
  \end{enumerate}
  the residual representation $\bar{\rho}_{A,\id{p}}$ is irreducible.
\end{thm}

Let $N^\pm$ be as in (\ref{eq:N-pm}). The following corollary is an application of Theorem~\ref{thm:irredasym}
to $\bar{\rho}^\pm_{\id{p}}$.

\begin{coro}
  \label{coro:large-image}
  There exists a constant $C(q,r)$ such that if $(a,b,c)$ is a
  primitive solution to (\ref{eq:GFE}) satisfying:
  \begin{enumerate}[(i)]
  \item $b$ is divisible by a prime $\ell \nmid N^{\pm}$,
  \item $r\nmid ab$ or $q\nmid bc$.
  \end{enumerate}
  Then, for every $p>C(q,r)$ and all $\id{p}\mid p$ in $K$, the residual representation $\bar{\rho}_\id{p}^\pm$ is absolutely irreducible.
  
\end{coro}
\begin{proof} Let $\id{q}$ and $\id{r}$ be prime ideals of $K$
  dividing $q$ and $r$ respectively. Let
  \[
    C(q,r) := \max\{C(K,\varphi(qr),\id{q}),C(K,\varphi(qr),\id{r}),q,r\},
  \]
 where the constants are those appearing in Theorem~\ref{thm:irredasym}.

  Let $(a,b,c)$ be a primitive solution to (\ref{eq:GFE}) and let $A$
  be the new part of $\Jac(\bfC^{\pm}(t_0))$.It is an abelian variety of $\GL_2$-type over $K$, of genus
  $\varphi(qr)=(q-1)(r-1)$. The variety $A$ is
  isogenous to the direct sum of the motive $\mot^{\pm}$ and its
  Galois conjugates. We need to verify that $A$ satisfies the
  hypotheses of Theorem \ref{thm:irredasym}.

  Suppose that $r\nmid ab$. Then Remark~\ref{remark:potentially-good}
  (and its analogue for primes larger than $3$) implies
  that $A$ has potentially good reduction at $\id{r}$. When
  $q \nmid bc$ a similar reasoning applies to the prime ideal
  $\id{q}$, proving condition $(i)$.

  Let $\id{p}$ be a prime ideal of $K$ dividing
  $p\neq q, r$. Theorems~\ref{thm:mot-+-prop} and \ref{thm:mot---prop} imply
  that if $\id{p} \nmid b$, $A$ has good reduction at $\id{p}$,
  otherwise, it has multiplicative reduction. In both cases it is
  semistable, so condition $(ii)$ is satisfied.
  
  The first hypothesis implies that there exists a prime
  $\ell\neq q,r$ (and also $\ell \neq 2$ in the case of $\mot^-$)
  dividing
  $b$. Let $\id{l}$ be a prime in $K$ dividing $\ell$. By
  Theorem~\ref{thm:mot-+-prop} (respectively
  Theorem~\ref{thm:mot---prop}) the prime $\id{l}$ is of
  multiplicative reduction for $\mot^+$ (respectively $\mot^-$), so
  the motive (and $A$) does not have complex multiplication and all
  its endomorphisms are defined over $K$.  The result now follows from
  Theorem \ref{thm:irredasym} and our definition of $C(q,r)$.
\end{proof}

\begin{remark} \label{rem:first-condition-irred} The first condition is used to ensure that $A$ does not have complex multiplication
	(see
  Proposition~\ref{prop:non-cm}). There are only finitely
  many specializations coming from non-trivial primitive solutions of
  the motive $\HGM^{\pm}(t)$ where the variety has complex
  multiplication (see \cite{MR971994} and \cite{MR1973057}). Then for $p$ large enough
  (with a bound that is ineffective) the first condition can be
  removed.
\end{remark}

\section{The equation $x^5 + y^p + z^3 = 0$}
In this section we apply the strategy developed in \cite{GP} and the results of
the previous sections to study primitive solutions to the equation
\begin{equation}\tag{\ref{eq:5p3}}
  x^5+y^p+z^3=0.
\end{equation}
\subsection{A hyperelliptic realization of the motive}
In \S \ref{section:prel} we explained how to obtain a curve whose
Jacobian realizes our motive $\HGM^{\pm}$. For $q=5$ and $r=3$, the quotient of Euler's curve $\bfC^+$ by the three
described involutions has genus $2$, hence is hyperelliptic and an
explicit equation for it can be computed. Euler's curve (\ref{eq:Euler+}) has equation
\begin{equation}
  \label{eq:euler-3p5}
  \bfC^+(t): y^{15} = x^2(1-x)^{-8} (1-tx)^{8}t^{-3}.    
\end{equation}
The curve has genus $14$. The old contribution to its Jacobian comes from the curves
\begin{equation}
  \label{eq:old-curves}
\bfC_5(t):y^5 = x^2(1-x)^{-8} (1-tx)^{8}t^{-3} \quad \text{and} \quad \bfC_3(t):y^3= x^2(1-x)^{-8} (1-tx)^{8}t^{-3},  
\end{equation}
of genus $4$ and $2$ respectively. The \emph{new} part of
$\Jac(\bfC^+)$ is an abelian variety of dimension $8$ with an action
of $\Z[\zeta_{15}]$ (of $\GL_2$-type over $F=\Q(\zeta_{15})$). To
compute quotients of curves by involutions in \texttt{Magma} we need
to use  a projective model. To keep degrees as small as
possible, we change variables $y \to y\cdot(1-x)$ and consider the projective model
\[
\bfC^+(t):y^{15}z^2=x^2(z-x)^7(z-tx)^8t^{-3}.
\]
Keeping the notation of \S \ref{section:prel}, an involution of the
curve (in the new variables) is given by the map
\[
  \iota_{-1}(x,y,z)=\Big(\frac{yz}{t},\frac{(x-z)(z-tx)}{t},xy\Big).
\]
We can define the involution in \texttt{Magma}
with the following commands:
\begin{verbatim} 
FF<t> := FunctionField(Rationals());
A<x,y,z> := ProjectiveSpace(FF,2);
C := Curve(A,y^15*z^2-x^2*(z-x)^7*(z-t*x)^8/t^3);
phiAmb := map<A->A|[y*z/t,(x-z)*(z-t*x)/t,x*y]>;
phi := Restriction(phiAmb,C,C);  
\end{verbatim}
Let $j=4$ (so $j \equiv -1 \pmod 5$ and $j \equiv 1 \pmod 3$). Then
the map (\ref{eq:inv1}) in the new variables is given by 
\[
\iota_4(x,y,z)=\left(\frac{x-z}{tx-z},\frac{(t-1)(z-x)^2(tx-z)x}{ty^4},1\right).
\]
It is another involution of $\bfC^+$. We define it in \texttt{Magma} by
\begin{verbatim}
phiAmb4 := map<A->A|[(x-z)/(t*x-z),(t-1)*(z-x)^2*(t*x-z)*x/t/y^4,1]>;
phi4 := Restriction(phiAmb4,C,C); 
\end{verbatim}
%
%
%
Then we can compute the quotient of $\bfC^+$ by both involutions (with
a little patience, since the computation takes around 10 minutes).
\begin{verbatim}
Mp := iso<C->C|DefiningEquations(phi),DefiningEquations(phi)>;
Mp4 := iso<C->C|DefiningEquations(phi4),DefiningEquations(phi4)>;
G := AutomorphismGroup(C,[Mp4,Mp]);
D,map := CurveQuotient(G);
\end{verbatim}
The output (available in \cite{github}) is a hyperelliptic curve isomorphic to 
\begin{equation}
  \label{eq:hyperell-5-p-3}
\bfC_{5,3}^+(t): y^2 = 5x^6 - 12x^5 + 10tx^3 + t^2.
\end{equation}
Its Jacobian over $K=\Q(\sqrt{5})$ has $\Q(\sqrt{5})$ as its algebra of
endomorphisms. There is a remarkable coincidence here: the quotient of
$\bfC^+$ by the group generated by the two involutions has genus
$2$. In general, only the new part of the quotient has genus $[K:\Q]$,
but in this case there is no contribution from the old parts
(corresponding to the quotients of the curves $\bfC_3$ and $\bfC_5$ in
(\ref{eq:old-curves}) by the same involutions). This is not the case
for the curve $\bfC^-$. We can compute the quotient curve, but it is a
curve of genus $5$. It is not clear to us how to get an equation for
the new part of it.

\medskip

In the recent paper \cite{ChenVillagra}, the authors study the
local conductor at the even place of the genus $2$
hyperelliptic curve 
\[
  C(t) : y^2 + y(x^3+t(1-t)^2) = 2t(1-t)^2x^3 + 3t^2(1-t)^3x + t^2(1-t)^4.
\]
The curve is related to solutions of the generalized Fermat equation 
\begin{equation}
  \label{eq:fermat-3-5-p}
x^3+y^5=z^p.  
\end{equation}
Concretely, if $(a,b,c)$ is a solution to (\ref{eq:fermat-3-5-p}),
then specializing at $t_0=a^3/c^p$ gives the curve
\[
  y^2 + y(x^3 + b) = 2bx^3 + 3ax + b^2,
\]
which seems to be suitable for the modular method. The hyperelliptic
curve $C(t)$ was obtained through a reverse engineering process, and
none of its expected properties were not addressed by the authors
\begin{prop}
  \label{prop:comparison}
  The hyperelliptic curves $\bfC_{5,3}^+((t-1)/t)$ and $C(t)$ are
  isomorphic.
\end{prop}
\begin{proof}
  Applying the change of variables $(x,y)\to{(1/x,y/x^3)}$ for the
  curve $\bfC_{5,3}^+(t)$, followed by the change of variable
  $y\to 2y + (tx^3+1)$, we get the curve
  \begin{equation*}
    y^2 + y(tx^3+1)=2tx^3-3x+1.
  \end{equation*}
  Substituting $t$ by $(t-1)/t$ and applying the change of variables
  $(x,y)\to(-x/(1-t),y/(t(1-t)^2))$ we obtain the model $C(t)$.
\end{proof}
\begin{remark}
  Theorem~\ref{thm:trace-frob} gives a formula to compute the trace of
  a Frobenius element $\Frob_{\id{p}}$ for $\id{p}$ a prime of
  $\Q(\zeta_{15})$ not dividing $15$ (not for primes in
  $K=\Q(\sqrt{5})$). Equation (\ref{eq:hyperell-5-p-3}) allows us to
  compute Frobenius traces over $K$.
\end{remark}
To lower the number of possibilities for the conductor of $\HGM^+$,
let us prove the veracity of Proposition~\ref{prop:congruences} over
$K$.
\begin{lemma}
  \label{lemma:congruences}
  Let $t_0 \in \Q\setminus\{0,1\}$. Then the following congruences hold over $K$.
  \begin{itemize}
  \item The motive
    $\HGM^+(t_0) \equiv
    \HGM((1,1)(\frac{1}{5},-\frac{1}{5})|t_0)\pmod{3}$.
  \item The motive
    $\HGM^+(t_0) \equiv
    \HGM((\frac{1}{3},-\frac{1}{3}),(1,1)|t_0)\pmod{(\sqrt{5})}$.
  \end{itemize}
\end{lemma}
\begin{proof}
  By Remark~\ref{remark:twist} the statement is true up to a quadratic
  twist by a character of $\Gal(\Q(\zeta_{15})/\Q(\sqrt{5}))$, i.e.\ up
  to one of the following characters: 
  \begin{itemize}
  	\item $\theta_{-3}$, corresponding to
  	the extension $\Q(\sqrt{-3},\sqrt{5})/\Q(\sqrt{5})$;
  	\item   $\theta_{\zeta_5}$, corresponding to the extension
  	$\Q(\zeta_{5})/\Q(\sqrt{5})$; 
  	\item $\theta_{6\sqrt{5}+15}$, corresponding to the third subextension obtained by adding to $\Q(\sqrt{5})$ a square root
  	of $6\sqrt{5}+15$.
  \end{itemize}
 
 Since the twist is independent of $t$, it is
  enough to prove the claim for a particular specialization of the
  parameter, say $t=3$. Recall that $\HGM^+(3)(1)$ (the Tate twist of
  the hypergeometric motive; see  Remark~\ref{rem:Tate-twist}) matches the hyperelliptic curve
  \[
    \bfC_{5,3}^+(3): y^2 = 5x^6 - 12x^5 + 30x^3 + 9.
  \]
  By (\ref{eq:ismorphism-inv}) we have that  $\HGM((1,1),(\frac{1}{5},-\frac{1}{5})|3)\simeq \HGM((\frac{1}{5},-\frac{1}{5}),(1,1)|\frac{1}{3})$. By Proposition \ref{prop:darmon-hyper-relation}, the latter 
  corresponds to Darmon's hyperelliptic curve
  \[
    C_5^+(1/3): y^2=x^6 + 6x^5 - 45x^4 - 270x^3 + 405x^2 + 2592x + 972.
  \]
 The values of the characters $\theta_{-3}$, $\theta_{\zeta_5}$, and
 $\theta_{6\sqrt{5}+15}$ at a prime ideal $\idK$ of $K$ dividing
 $\{7,11,29\}$ are, respectively, $(+,-,-)$, $(-,+,-)$, and $(-,-,+)$.
 Therefore, if we prove that the congruence holds for $\Frob_{\idK}$ at
 each such $\idK$, then it cannot hold for any of the possible twists,
 yielding the desired result.
 
  The characteristic polynomials $P_\idK$ of $\Frob_{\idK}$
  for both curves $\bfC_{5,3}^+(3)$ and $C_5^+(1/3)$ are given in
  Table~\ref{table:char-pol} (they were computed using
  \texttt{Magma}).
  \begin{table}[H]
    \label{table:char-pol}
  \begin{tabular}{c|c|c}
    $\idK$ & $P_{\idK}(\bfC_{5,3}^+(3))$ & $P_{\idK}(C_5^+(1/3))$\\
    \hline
    $7$ & $x^4 + 10x^3 + 123x^2 + 490x + 2401$ & $x^4 - 20x^3 + 198x^2 - 980x +   2401$\\
    \hline
    $(\frac{1+3\sqrt{5}}{2})$ & $x^4 - x^3 + 21x^2 - 11x + 121$ & $x^4 + x^3 + 21x^2 + 11x + 121$\\
    \hline
    $(\frac{11+\sqrt{5}}{2})$ &$x^4 - 2x^3 + 14x^2 - 58x + 841$ & $x^4 + 5x^3 + 53x^2 + 145x + 841$\\
    \hline
  \end{tabular}
  \caption{Characteristic polynomials of Frobenius.}
  \end{table}
  Recall that we need to add a Tate twist to the second column or,
  equivalently, replace $x$ by $\norm(\idK)x$. The twisted polynomials
  of the second column match those of the third column modulo $3$,
  proving the first statement. To prove the second statement, we use
  that, by Proposition \ref{prop:darmon-hyper-relation},
  $\HGM((\frac{1}{3},-\frac{1}{3}),(1,1)|3)$ corresponds to the
  elliptic curve
  \[
   C_3^+(3)\simeq E_3^+(3):y^2+3xy+3y=x^3.
  \]
  The characteristic polynomial of the Frobenius endomorphism $\Frob_{\idK}$ acting on
  $E_3^+(3)$ for the same prime ideals $\idK$ as before are
  \[
    x^2 + 10x + 49, \qquad x^2 - 3x + 11, \qquad x^2 + 6x + 29.
  \]
 The square of each of these polynomials matches the twisted polynomials
 in the second column of Table~\ref{table:char-pol} modulo $5$. See the
 script in \cite{github} for a verification of these computations.
\end{proof}

\begin{prop}
  \label{prop:Igusa}
  Let $\ell\in\{3,5\}$ and $t \in \QQ\setminus \{0,1\}$. If $v_\ell(t-1)> 3$
  the Jacobian of $\bfC_{5,3}^+(t)$ has potentially multiplicative
  reduction at $\ell$. Moreover, if $v_\ell(t-1)=0$ and also
  $v_3(17t+108)\le 3$ when $\ell=3$, then it has potentially good
  reduction at $\ell$.
\end{prop}

\begin{proof}
The Igusa invariants of $\bfC_{5,3}^+(t)$ (see  \cite[page 623]{MR114819}) are given by 
\begin{align*}
	J_2 &= -1200t^2,\\
	J_4 & =-480000t^4,\\
	J_6 & = 43520000t^6 + 276480000t^5,\\
	J_8 &= -70656000000t^8 - 82944000000t^7,\\
	J_{10} & = 2388787200000t^{10} - 4777574400000t^9 + 2388787200000t^8.
\end{align*}
Set $\varepsilon=1$ if $\ell=5$ and $\varepsilon=3$ if $\ell=3$.
By Theorem 1 of \cite{Liu}, the Jacobian of $\bfC_{5,3}^+(t)$ has
potentially good reduction if and only if one of the two following conditions holds:
\begin{itemize}
\item [(I)] The values $J_{2i}^5J_{10}^{-i}$ are integral
  over $\Q_\ell$ for $1\le i\le 5$,
  
\item [(V)] All the following conditions hold:
  $I_4^\varepsilon I_{2\varepsilon}^{-2} \in \ell\Z_\ell$,
  $J_{10}^\varepsilon I_{2 \varepsilon}^{-5} \in \ell\Z_\ell$,
  $I_{12}^\varepsilon I_{2\varepsilon}^{-6} \in \ell\ZZ_\ell$,
  $I_4^{3\varepsilon}J_{10}^{-\varepsilon}I_{2\varepsilon}^{-1} \in
  \ZZ_\ell$ and
  $I_{12}^\varepsilon J_{10}^{-\varepsilon}I_{2\varepsilon}^{-1} \in
  \ZZ_\ell$, where $I_4=J_2^2-24J_4$, $I_2=J_2/12, I_6=J_6$ and $I_{12}=-2^3J_4^3+3^2J_2J_4J_6-3^3J_6^2-J_2^2J_8$.
\end{itemize}
The first case corresponds to a genus $2$ curve with good reduction,
while the second case corresponds to a genus $2$ curve whose special
fiber consists of two elliptic curves with complex multiplication
meeting at a single point (so the curve has bad reduction but its
Jacobian does not).

The values of (I) are respectively
\[
  -\frac{5^5t^2}{3(t-1)^2}, \quad -\frac{5^{10}t^4}{3^7(t-1)^4},\quad \frac{5^5t(17t+108)^5}{3^{18}(t-1)^6}, \quad -\frac{5^{10}t^3(23t+27)^5}{3^{19}(t-1)^8}, \quad 1.
\]
The values of (V) for $\ell=3$ are respectively
{\small\[
\frac{3^{12}5^4t^2}{(17t+108)^2},\quad \frac{3^{18}(t-1)^6}{5^5t(17t+108)^5},\quad \frac{3^{27}(t-1)^3(9t+16)^3}{(17t+108)^6},\quad \frac{3^{18}5^{17}t^7}{(17t+108)(t-1)^6},\quad \frac{3^95^5t(9t+16)^3}{(17t+108)(t-1)^3}.
\]}
For $\ell=5$, they equal
{\small\[
6^4 ,\quad -\frac{2^{10}3^6(t-1)^2}{5^5t^2},\quad \frac{2^{12}3^9(t-1)(9t+16)}{5^4t^2},\quad -\frac{2^23^65^5t^2}{(t-1)^2},\quad -\frac{2^23^35(9t+16)}{t-1}.
\]}
All of these computations can be checked in \cite{github}. Conditions (I) and (V) do not hold at $\ell=3, 5$ if $v_\ell(t-1)> 3$,
  proving the first claim. If $v_\ell(t-1)=0$ then condition (I) holds
  for $\ell=5$, while condition (V) holds for $\ell=3$, since $v_3(17t+108)\le 3$ by hypothesis.
\end{proof}
\begin{remark}\label{rem:pot-good-red}
  If $(a,b,c)$ is a hypothetical solution to~(\ref{eq:5p3}) and
  $t_0=-\frac{a^5}{c^3}$ then it is always the case that
  $v_3(17t_0+108)\le 3$. Indeed, if
  $17t_0+108 = \frac{3^n\alpha}{\beta}$, with $n \ge 4$ and
  $3 \nmid \beta$ then
  $t_0 = \frac{3^3(3^{n-3}\alpha-4\beta)}{17\beta}$. Since the
  numerator is a fifth power, $3 \mid \beta$, giving a contradiction. This implies that the
  last condition of Proposition \ref{prop:Igusa} is automatically satisfied for our Diophantine applications.
\end{remark}

\begin{coro}\label{coro:Curve-with-non-CM}
  Let $(a,b,c)$ be a non-trivial primitive solution to \eqref{eq:5p3} with
  $p>3$ and set $t_0=-\frac{a^5}{c^3}$. Then the Jacobian of the
  specialization $\bfC_{5,3}^+(t_0)$ does not  have complex nor quaternionic
  multiplication.
\label{coro:non-cm}
\end{coro}

\begin{proof}
  It follows from Proposition~\ref{prop:non-cm} that $b$ can only be
  divisible by the primes $3$ and $5$. If $3 \mid b$ (respectively
  $5 \mid b$) then Proposition \ref{prop:Igusa} and  Remark \ref{rem:pot-good-red} implies that
  $\Jac(\bfC_{5,3}^+(t_0))$ has potential multiplicative reduction at
  $3$ (respectively 5), so it cannot have complex nor quaternionic multiplication (see \cite[\S 4]{Jordan}). Then $b = \pm 1$, and
  $(\mp a,\mp c)$ must be a solution to Catalan's equation
  \[
   x^5 - (-z)^3 = 1,
 \]
 which does not have any non-trivial solution (as proved in \cite{MR2076124}).
\end{proof}
The existence of a hyperelliptic model allows us to prove a stronger
irreducibility result.
\begin{coro}
  \label{coro:large-image-5p3}
  Let $\ell$ be a fixed prime. There exists a bound $C(\ell)$ such
  that the following holds.  If $(a,b,c)$ is a non-trivial primitive
  solution to (\ref{eq:5p3}) satisfying that $\ell\nmid b$ then for
  all primes $p > C(\ell)$ and all $\id{p}\mid p$ in $K$, the residual
  representation of ${\rho}_\id{p}^\pm$ is absolutely irreducible.
   Explicitly, $C(2)=13$, $C(3)=41363281$ and $C(5)=335809$.
\end{coro}
\begin{proof}
  The proof follows the one given in Corollary~\ref{coro:large-image}.
  By Corollary \ref{coro:Curve-with-non-CM}, the Jacobian
  $J=\Jac(\bfC_{5,3}^+(t_0))$ does not have complex multiplication, so
  the first hypothesis of Corollary~\ref{coro:large-image} can be
  removed (see Remark
  \ref{rem:first-condition-irred}).
  Let $\idK$ be a prime in $K$ above $\ell$. Taking $\idK$ in
  Theorem~\ref{thm:irredasym} we need to verify that the variety $J$
  has potentially good reduction at $\idK$. Since $\ell\nmid b$,
  this follows from Theorem \ref{thm:mot-+-prop} (1) if $\ell\neq 3,5$
  and from Proposition \ref{prop:Igusa} if $\ell=3,5$.  The explicit
  value of $C(\ell)$ can be computed following the proof of Theorem 2
  in \cite{MR3690598}. Let $M$ denote a field where $J/K$ attains
  good reduction at $\idK$, and let $f$ denote the inertial degree
  of $M/K$. Then according to loc.\ cit.,
  \begin{equation}
    \label{eq:reducible-bound}
    p \mid \prod_{a_J} \text{Res}(X^{c_J \cdot h_K}-1,x^2-a_Jx+\norm(\idK)^f),    
  \end{equation}
  where Res denotes the resultant, $c_J$ is the inertial exponent of $J$, $h_K$
  denotes the class group of $K$ and the product runs over all the
  possible local factors of $J$ over $M$. 
  
  In our case, $h_K=1$. If $\ell=3,5$ then  $f=1,2$ and a priory $c_J=4, 12$ for $q=3$ and $c_J=4,20$ for
  $q=5$. Otherwise, $\ell\neq 3,5$ and so, since $J$ has good reduction over $K$, we have $f=c_J=1$.
  
  Suppose $\ell=2$. Then, an easy computation in \texttt{Magma} shows
  that $a_J\in\{-1,-8\}$. Then, the product in
  (\ref{eq:reducible-bound}) equals 6084 and so $C(2)=13$.
  
  Suppose then that $\ell=3,5$ and that the residual
  representation decomposes like
  \[
   \bar{\rho}_{\id{p}} \sim \psi_1\chi_p \oplus \psi_2,
 \]
 where $\chi_p$ denotes the cyclotomic character. Then $\psi_1,\psi_2$
 have conductor dividing $3\cdot(\sqrt{5})$. The ray class
 group 
 of $K$ of modulus $3(\sqrt{5})\infty_1\infty_2$, where $\infty_i$
 denote the two places at infinity, is isomorphic to
 $\Z/4\Z \times \Z/2\Z$. Hence $\psi_i$ has order at most $4$, for $i=1,2$. Then
 (following the proof of \cite[Theorem 2]{MR3690598}) we can take
 $c_J=4$ in (\ref{eq:reducible-bound}). To compute all possible values
 of $a_J$ in (\ref{eq:reducible-bound}), recall that
 $a_J \in \Z[\frac{1+\sqrt{5}}{2}]$. Write
 $a_J = \frac{\alpha + \beta\sqrt{5}}{2}$. Then the Hasse-Weil bound (our
 curve has genus $2$) over $\F_\ell$ and $\F_{\ell^2}$, where
 $\ell = \normid{\idK}$, gives the inequalities
 \[
   |\alpha| \le 4\sqrt{\ell}, \qquad \left|\frac{\alpha^2+5\beta^2}{2}-4\ell\right|\le 4\ell.
 \]
 This gives the bounds for $f=1$: $C(3)=8969$, $C(5)=809$ and for
 $f=2$: $C(3)=41363281$, $C(5)=335809$. The code to compute these
 values is included in \cite{github}.
\end{proof}

\subsection{Spaces of Hilbert modular forms}

Let $K=\Q(\sqrt{5})=\Q(\zeta_5)^+$. Let $(a,b,c)$ be a primitive solution to
\begin{equation}\tag{\ref{eq:5p3}}
	x^5 + y^p + z^3 =0.
\end{equation}
Keeping the notation of \S\ref{section:prel}, let
$\rho_{\id{p}} := \rho_{\id{p}}^{+} : \Gal_K \to \GL_2(K_{\id{p}})$
be the Galois representation in (\ref{eq:gal-rep}) attached to the
specialization of the motive $\HGM^{+}$ at the point
$t_0 = -\frac{a^{5}}{c^{3}}$.
 Define $a_0, c_0 \in \ZZ$ by
$a_0=a 5^{-v_5(a)}$ and $c_0=c 3^{-v_3(c)}$. Let $\id{p}$ be a prime
in $K$ above $p$.

%
%

\begin{thm}
  \label{thm:spaces}
  Suppose that $\bar{\rho}_{\id{p}}$ is irreducible.  Then, there
  exists a Hilbert newform
  $\newform \in S_2(\Gamma_0(3^{\varepsilon_3}\cdot
  (\sqrt{5})^{\varepsilon_5}))$ over $K$ such that
  \begin{align}\label{eq:iso-modform}
    \bar{\rho}_\id{p} \simeq\bar{\rho}_{\newform,\id{P}},
  \end{align}
  where $\id{P}$ is a prime ideal of the coefficient field $K_\newform$
  dividing $p$. In addition,
  \begin{enumerate}[(i)]
  \item The value $\varepsilon_3$ satisfies
    \[
      \varepsilon_3=\begin{cases}
			2 & \text{if } 3\mid ab,\\ 
			2 &  \text{if } 3\mid c  \text{ and } \pm2 c_0^3\equiv a^5\pmod 9,\\
			2 & \text{if } 3\nmid abc  \text{ and } 4c^3\equiv a^5\pmod 9,\\
			3 & \text{otherwise.}
                    \end{cases}
   \]
 \item The value $\varepsilon_5$ satisfies
   \[
     \varepsilon_5=\begin{cases}
			2  & \text{if } 5\mid bc,\\
			2 & \text{if } 5\mid a \text{ and } \frac{c^3}{a_0^5}\equiv 6, 12, 28 \pmod {25},\\
			2 & \text{if } 5\nmid  abc \text{ and } \frac{c^3}{a}\equiv 6,8, 17, 19 \pmod {25}, \\
			3 & \text{otherwise.}
                   \end{cases}	
  \]
\end{enumerate}
\end{thm}
\begin{proof} 
 By  Theorem~\ref{thm:modularity} and \cite[Theorem 3.13]{BCDF1}, 
 we have that  $\bar{\rho}_{\id{p}}$  is modular with trivial Nebentypus. Since $K$ is totally real, $\bar{\rho}_{\id{p}}$ is absolutely irreducible.  Hence, by level lowering for Hilbert modular forms \cite{MR3294620}, we have that (\ref{eq:iso-modform}) holds for some newform $\newform$  of parallel weight 2, trivial Nebentypus over $K$ and level equals the conductor of $\bar{\rho}_{\id{p}}$. By
 Theorem~\ref{thm:mot-+-prop} the latter is only supported at the prime ideals $(3)$ and $(\sqrt{5})$.
 
 To compute the value $\varepsilon_3$ we use Corollary
 \ref{coro:conductor-at-3} and Remark~\ref{remark:non-twist}
 (Lemma~\ref{lemma:congruences} says that there is no twist involved
 in the congruence). When $3 \mid a$, the curve $\bfC^+_{5,3}(t_0)$
 has potential good reduction (by Proposition~\ref{prop:Igusa} and
 Remark \ref{rem:pot-good-red}) so the conductor exponent must be
 $2$. The value of $\varepsilon_5$ when $5 \mid bc$ follows from
 Corollary \ref{coro:conductor-at-r} (and
 Remark~\ref{rem:twist-control}) and the result when $5 \mid a$
 follows from Proposition \ref{prop:Igusa}. To prove the last two
 cases we need to characterize when the polynomial $F(x)$ is
 irreducible over $\Q_5$. By Lemma \ref{lemma:irreducibility} there
 are finitely many cases to consider, so we run all cases with a 
 script available in \cite{github}.
\end{proof}
Table~\ref{table:dimensions} contains information on the spaces of
cuspidal Hilbert modular forms of parallel weight $2$ and level
$\Gamma_0(3^i\cdot(\sqrt{5})^j)$ for $i, j \in \{2,3\}$. We included
an extra column with the number of newforms whose coefficient field
equals $\Q(\sqrt{5})$. The script used to compute the entries of the
table is available at \cite{github}.

\begin{table}[H]
  \begin{tabular}{|l|c|c|c|}
    \hline
    Level &  Dimension & \# newforms & $K_\newform=K$ \\
    \hline
    \hline
    $3^2\cdot(\sqrt{5})^2$ & $45$ & $14$ & $4$\\
    \hline
    $3^3\cdot(\sqrt{5})^2$ & $405$ & $111$ & $11$\\
    \hline
    $3^2\cdot(\sqrt{5})^3$ & $225$ & $35$ & $21$\\
    \hline
    $3^3\cdot(\sqrt{5})^3$ & $2025$ & $112$ & $44$\\
    \hline
    \end{tabular}
    \caption{Dimensions of spaces of Hilbert modular forms.}
    \label{table:dimensions}
\end{table}

\subsection{Ghost solutions}\label{subsec:ghost-5p3}
Recall from Definition \ref{def:ghost} that ghost solutions are tuples of the form $(a,c,s,\ell,m,n,u,v)$, $ac\neq0$, satisfying 
\begin{equation}
  \label{eq:ghost}
  5^s3^\ell x^5 \pm 5^m 3^n + 5^u3^v z^3=0,
\end{equation}
that might not be solutions to (\ref{eq:5p3}). 
Following \S \ref{section:obstruction}, in this case we set 
$t_0:=-\frac{ a^5 5^{s-u}3^{\ell-v}}{c^3}$. Doing
a computer search we found the list of solutions given in
Table~\ref{table:ghost-sols} (which is probably complete). The table
includes the conductor of the $2$-dimensional representation obtained
from this specialization.
\begin{table}[H]
  \begin{tabular}{|r|c||r|c||r|c||r|c||r|c|}
    \hline
    $t_0$ & Conductor & $t_0$ & Conductor& $t_0$ & Conductor & $t_0$ & Conductor& $t_0$ & Conductor\\
    \hline
    \hline
    $-\frac{32}{343}$ & $3^4 \cdot(\sqrt{5})^2$ & $\frac{9}{8}$ & $3^3 \cdot(\sqrt{5})^3$& $-\frac{1}{8}$ & $3^3\cdot(\sqrt{5})^3$ & $\frac{25}{24}$ & $3^5\cdot(\sqrt{5})^7$ & $-\frac{1}{24}$ & $3^5\cdot(\sqrt{5})^3$ \\
    \hline
    $\frac{32}{5}$ & $3^2\cdot(\sqrt{5})^5$ & $\frac{3}{8}$ & $3^4\cdot(\sqrt{5})^5$ & $\frac{32}{27}$ & $3^3\cdot(\sqrt{5})^5$ & $\frac{5}{8}$ & $3^4\cdot(\sqrt{5})^7$ & $\frac{1024}{1029}$& $3^5\cdot(\sqrt{5})^5$\\
    \hline
    \end{tabular}
      \caption{Ghost solutions.}
      \label{table:ghost-sols}
\end{table}	
The values $t_0=-\frac{1}{8}$ and $t_0=\frac{9}{8}$ (coming from Catalan's solution) 
are the only ones that correspond to newforms in
$S_2(\Gamma_0(3^3 \cdot (\sqrt{5})^3))$. As we will prove, they
correspond to the two newforms in the space that cannot be discarded
using Mazur's method, which is the standard strategy to eliminate
possible newforms satisfying (\ref{eq:iso-modform}). Note that the
study of ghost solutions allows us to find an equation for the
hyperelliptic curve attached to these two newforms. This might be
useful while applying other elimination procedures, like the
symplectic argument.

\subsection{Discarding solutions: Mazur's method}
\label{section:Mazur}
To discard forms not related to solutions, we apply an idea due to
Mazur that we now recall. Let $\newform$ be one of the newforms in the
above described spaces. For $\idK$ a prime ideal of $K=\Q(\sqrt5)$,
denote by $a_{\idK}(\newform)$ the $\idK$-th Fourier coefficient
of $\newform$.

The underlying idea of the method is to understand solutions to (\ref{eq:5p3})
modulo $\ell$ (for $\ell$ a rational prime not dividing $15p$). They
will impose congruence conditions on the Fourier coefficient
$a_{\idK}(\newform)$ for prime ideals $\idK$ of $K$  dividing
$\ell$. More concretely, let
\[
  S_\ell = \{(\tilde{a},\tilde{b},\tilde{c}) \in
  \F_\ell^3\setminus\{(0,0,0)\} \; : \;
  \tilde{a}^5+\tilde{b}^p+\tilde{c}^3=0\}.
\]
Any primitive solution $(a,b,c)$ to (\ref{eq:5p3})
reduces to an element in $S_\ell$. Let $S_\ell^\times$ denote the subset of
$S_\ell$ made of elements where none of their entries are zero. Then a
solution $(a,b,c)$ belongs to one of the three cases:
\begin{enumerate}[\text{Case} (1).]
\item Its reduction $(\tilde{a},\tilde{b},\tilde{c}) \in S_\ell^\times$.
\item Either $\tilde{a} = 0$ or $\tilde{c}=0$, but $\tilde{b} \neq 0$.
\item  $\tilde{b}=0$.
\end{enumerate}
\subsubsection{Case $(1)$:} The first case corresponds to values
$t_0=-\frac{a^5}{c^3}\in\PP^1(\Q)\setminus \{0,1,\infty\}$ whose mod $\ell$ reduction $\tilde{t}_0$ belongs to $\F_\ell\setminus \{0,1\}$. We compute for each  $\tilde{t}_0 \in \F_\ell\setminus \{0,1\}$
 the value $a_{\idK}(\mot^+)$ corresponding to
the trace of $\Frob_{\idK}$ acting on $\mot^+$, using the code
written by Fernando Rodriguez Villegas  in \texttt{Pari/GP} (available at the \texttt{GitHub}
repository \url{https://github.com/frvillegas/frvmath}). If the newform $\newform$ is attached to a solution,
then
\begin{equation}
  \label{eq:Mazur}
a_{\idK}(\newform) \equiv a_{\idK}(\mot^+) \pmod{\id{p}},  
\end{equation}
where $\id{p}$ is a prime ideal dividing $p$ in the composition of $K_\newform$ and $K$. 

\begin{remark}
  There is a caveat here: Villegas' code computes the trace of a
  Frobenius element as an $\ell$-adic number. The number is in fact
  algebraic  (see \cite[Corollary 6.19]{GPV}). In practice, we
  compute the number to enough precision that allows us to recognize
  it as a root of a degree $2$ polynomial (without distinguishing
  between the value and its Galois conjugate).
 
\end{remark}

\begin{remark}
  For the particular signature $(5,p,3)$ we have a genus 2 curve
  giving rise to the hyperelliptic motive so we can compute its Euler
  factor using \texttt{Magma}. However, we deliberately choose to use
  the aforementioned code to compute traces of hypergeometric motives,
  because our implementation works for any signature, opening the door to
  future applications with signatures for which no Frey hyperelliptic
  curve is known.
\end{remark}
We say that a prime $\ell$ \textit{eliminates} Case (1) of a newform $\newform$
if there exists a prime $\idK\mid \ell$ of $K$ such that the
left-hand side of (\ref{eq:Mazur}) is different from the right-hand side
for all possible specializations $\tilde{t}_0 \in \F_\ell\setminus \{0,1\}$. In this case we get an
explicit bound for the exponents $p$ for which a solution might exist,
since $\id{p}$ must divide the (non-zero) value
$a_{\idK}(\newform) - a_{\idK}(\mot^+)$.

We say that a prime $\ell$ \textit{fails} Case (1) for $\newform$ if
for each prime ideal $\idK \mid \ell$ there is a value $\tilde{t}_0$
such that both sides coincide. Catalan's solutions described in \S
\ref{section:obstruction} correspond to two Hilbert newforms in
$S_2(\Gamma_0(3^3\cdot(\sqrt{5})^3))$ that give for each prime $\ell$
two values $\tilde{t}_0$ (the reduction of $-\frac{1}{8}$ or
$\frac{9}{8}$ modulo $\ell$) for which Case (1) fails.

\begin{remark}
  If the coefficient field $K_\newform$ is different from $K$, then
  there is a prime ideal $\idK$ that eliminates Case~(1) of
  $\newform$. This argument applies to all three cases, so while
  proving non-existence of solutions to (\ref{eq:5p3}) for $p$
  sufficiently large, it is enough to restrict to newforms $\newform$
  satisfying $K_\newform =K$. This justifies the last column of
  Table~\ref{table:dimensions}.
\end{remark}

\begin{exm}
  Take $\ell=11$, a split prime in $K$ and let
  $\idK=(\frac{-1+3\sqrt{5}}{2})$. We need to compute the minimal
  polynomial of $11 \cdot a_{\idK}(\mot^+)$ (the Tate twist
  of the motive, as explained in Remark~\ref{rem:Tate-twist}) for
  $\tilde{t}_0=2,\ldots,6$. Here is how we compute the value for
  $\tilde{t}_0=2$ with a $10$ digits precision. 
   {\small
\begin{verbatim}
? hgm(2,[1/5,-1/5],[1/3,-1/3],11,1,10)
% 1 = 9*11^-1 + 7 + 4*11^2 + 11^3 + 4*11^4 + 2*11^5 + 3*11^6 + 5*11^7 + 7*11^8
+ O(11^9)
\end{verbatim}
  } The input of the routine is as follows: the value $\tilde{t}_0$, the parameters, the
  prime $\ell$, the inertial degree of $\ell$ in $K$ (in this case
  $1$) and the precision.  We  can recognize this value as an algebraic integer as follows.
\begin{verbatim}
? algdep(hgm(2,[1/5,-1/5],[1/3,-1/3],11,1,10)*11,2)
%2 = x + 5
\end{verbatim}
  There is a caveat here: the script that we used does not depend on
  the prime ideal $\idK$, but of its norm! The underlying reason is
  that the finite hypergeometric sum is really a number in
  $K=\Q(\sqrt{5})$. If we want to describe it as an element in
  $\Q_{11}$ we need to choose a prime ideal of $K$ dividing
  $11$. Different choices give algebraic conjugate values, so while
  recognizing the $11$-adic number as an algebraic number, we do not
  distinguish between the values at both primes dividing $11$, but in practice
  this is still good enough for the elimination procedure.
  
  Applying the algorithm to the other values of $\tilde{t}_0$ we obtain the
  values in Table~\ref{table:aps}.
  \begin{table}[H]
    \begin{tabular}{|c|c||c|c||c|c|}
      \hline
      $t_0$ & $11a_{\idK}(\mot^+)$ & $t_0$ & $11a_{\idK}(\mot^+)$ & $t_0$ & $11a_{\idK}(\mot^+)$\\
      \hline
      \hline
      $2$ & $x+5$ & $5$ & $x-2$ &  $8$ & $x^2+2x-19$\\
      \hline
      $3$ & $x^2-5$ &  $6$ & $x^2+4x-1$  &$9$ & $x^2+3x+1$\\
      \hline
      $4$ & $x^2-x-1$ & $7$ & $x^2+5x-5$ & $10$ & $x^2+x-11$ \\
      \hline
    \end{tabular}
      \caption{Characteristic polynomial of $11a_{\idK}(\mot^+)$.}
      \label{table:aps}
\end{table}	
%
The space $S_2(\Gamma_0(3^2\cdot (\sqrt{5})^2))$ has $14$ (Galois orbits of) 
newforms whose $\idK$-th eigenvalue equals:
$0, -2, 0, 4, -2, 0,$ $\pm 3\sqrt{2}, \pm 2\sqrt{2}, \frac{\pm 9+\sqrt{5}}{2}, \frac{\pm 3+3\sqrt{5}}{2}$. These values are
different from the roots of the polynomials in Table~\ref{table:aps}, proving that the form
cannot correspond to a Case (1) solution if $p>281$. Using more values for the prime $\ell$ we get a better bound.
\end{exm}

\subsubsection{Case $(2)$:} The second case corresponds to a prime
$\ell$ dividing $ac$ (so $t_0$ reduces to $0$ or $\infty$ modulo
$\ell$). Since $\ell\nmid b$, by Theorem \ref{thm:mot-+-prop}, the
motive $\mot^+$ has good reduction at any prime ideal of $K$ dividing
$\ell$, but Euler's curve has bad reduction at it. Using the formula
given in \cite[Appendix A ]{GPV} we can compute the trace of a
Frobenius element for a prime ideal $\idF$ of $F=\Q(\zeta_{15})$
dividing $\ell$.  Write $t_0= \ell^{v_\ell(t_0)}
\widetilde{t_0}$. Then if $t_0$ reduces to $0$ modulo $\ell$, the
trace of $\Frob_{\idF}$ is given by the formula
\begin{equation}
  \label{eq:trace-0}
    - \norm(\idF)^{-1}\cdot 
    \left(\gent(\widetilde{t_0})^{-3}J(\gent^{2},\gent^{-8})+\gent(\widetilde{t_0})^{3}J(\gent^{-6},\gent^{8})\right),    
\end{equation}
where $\chi_\idF$ is as in (\ref{eq:char-def}) and $J(\cdot,\cdot)$ denotes the usual Jacobi sum. If $t_0$
reduces to $\infty$, then the trace of $\Frob_{\idF}$ equals
\begin{equation}
  \label{eq:trace-oo}
  -\norm(\idF)^{-1}\cdot 
  \left(\gent(\widetilde{t_0})^{-5}J(\gent^{2},\gent^{8})+\gent(\widetilde{t_0})^{5}J(\gent^{10},\gent^{-8})\right).
\end{equation}
As in the previous case, the value of  $\widetilde{t_0}$ mod $\ell$ is a priori any
element of $\F_\ell^\times$.
\begin{remark}
  The character $\gent^3$ has order $5$ so (\ref{eq:trace-0}) takes at
  most five different values. Similarly, the character $\gent^5$ has
  order $3$, so (\ref{eq:trace-oo}) takes at most $3$ different
  values.
\end{remark}

\begin{exm}
  \cite{github} contains a
  \texttt{Pari/GP} file with different scripts. One of them allows to compute
  the set of possible values for (\ref{eq:trace-0}) and (\ref{eq:trace-oo}).   Take as in the
  previous example $\ell=11$ and let us show how to compute (\ref{eq:trace-0}). The five different values are computed by
\verbatimfont{\footnotesize}
\begin{verbatim}
? Degenerate0(11)
%1 = [x^2 - 4*x - 316, x^2 + x - 101, x^2 - 19*x - 61, x^2 - 19*x + 59, x^2 + 41*x + 419]
\end{verbatim}
\verbatimfont{\small} Note that this gives the characteristic
polynomial of the possible values of the Frobenius coefficient over
the field $F$. We can compute the possible values over $K$ using
\texttt{Magma} scripts for the hyperelliptic curve
$\bfC_{5,3}^+$. Concretely, let $t=11^5$ and $\idK=(
4+\sqrt{5})$. Then we get {\small
\begin{verbatim}
_<x> := PolynomialRing(Rationals());
K := NumberField(x^2-5);
O := Integers(K);
Id := Factorisation(11*O)[1][1];
EulerFactor(BaseExtend(HyperellipticCurve(5*x^6-12*x^5+10*11^5*x^3+11^10),K),Id);
121*x^4 - 44*x^3 + 6*x^2 - 4*x + 1
\end{verbatim}
} Over the field $K$ the degree $4$ polynomial factors in the form
  \[(11x^2 -(2+2\sqrt{5})x + 1)(11x^2 - (2-2\sqrt{5})x + 1),\] so the
  value $a_{\idK}(\mot)$ 
  corresponds (after normalizing by multiplication by $\norm(\idK)=11$)
  to one of the values $2 \pm 2\sqrt{5}$. The prime $\idK$ has
  inertial degree $2$ over $F$, so the trace of Frobenius at a prime
  dividing $\idF$ equals $a_{\idF}=a_{\idK}^2-22$, which is a root of
  $x^2 - 4x - 316$ (the first polynomial).
\end{exm}
To discard Case (2) solutions, we compute all possible values
(\ref{eq:trace-0}) and (\ref{eq:trace-oo}) using our script and
compare them to the values $a_{\idK}(\newform)$ for each newform
$\newform$. If they are different, then the ideal $\id{p}$ must divide
their difference.

\subsubsection{Case $(3)$:} This case corresponds to the usual
\emph{level lowering} situation, namely the representation
attached to the motive $\mot^+$ has multiplicative reduction at the
prime ideal $\idK$ but the residual $\id{p}$-th adic one does
not. If $\newform$ is an eigenform in this situation, the following
congruence must hold
\begin{equation}
  \label{eq:lower-level}
  a_{\idK}(\newform) \equiv \pm (\norm(\idK) + 1) \pmod{\id{p}}.
\end{equation}

\vspace{5pt}

The described procedure allows us to prove the following result.

\begin{thmA}\label{thmA}
 Set
  $\bfP:=\{2,3,5,7,11,13,19,29,31,41,61,71,79,89,101,109\}$. Let $p$ be a prime
  number not in $\bfP$ and let $(a,b,c)$ be a primitive solution to
  \begin{equation}\tag{\ref{eq:5p3}}
    x^5 + y^p + z^3 =0.
  \end{equation}
  Set $t_0:=-a^5/c^3$. Then one of the following holds:
  \begin{enumerate}
  \item The residual Galois representation
    $\overline{\rho}_{\id{p}}^+$ is reducible.
  \item The Galois representation $\rho_{\id{p}}^+$ is congruent
    modulo $p$ to the Galois representation associated to a Hilbert modular form of $\Q(\sqrt{5})$ with
    complex multiplication by $\Q(\sqrt{-3})$ or by
    $\Q(\zeta_{15})/\Q(\sqrt{5})$.
  \item The Galois representation $\rho_{\id{p}}^+$ is congruent
    modulo $p$ to ${\rho}_{\bfC_{5,3}^+(t_0),\id{p}}$, where $t_0\in\{-\frac{1}{8}, \frac{9}{8}\}$ (these two values correspond to ghost solutions).
	\end{enumerate}
\end{thmA}
\begin{proof}
  Let $\rho^+_{\id{p}}$ be the representation of $\HGM^+$
  specialized at $t_0$, as in (\ref{eq:gal-rep}). Suppose that
  $\bar{\rho}^+_{\id{p}}$ is irreducible. Then by Theorem
  \ref{thm:spaces}, there exists a Hilbert newform
  $\newform \in S_2(\Gamma_0(3^i\cdot (\sqrt{5})^j))$, with
  $i,j \in \{2,3\}$, such that
  $\bar{\rho}_{\id{p}}\simeq \bar{\rho}_{\newform,\id{P}}$, where
  $\id{P}$ is a prime in $K_\newform$ dividing $p$.
  
  We apply an implementation in
  \texttt{Magma} of the method previously described as follows: first
  we chose a set of primes to apply the elimination procedure. The
  stated result corresponds to the choice all rational primes $\ell$
  not dividing $30$ having a prime ideal $\idK$ in $K$ dividing it
  with $\norm(\idK) \le 400$. 
  For each prime
  $\idK$ we compute the possible value of $a_{\idK}(\HGM^+(t))$ for all
  values of $t$ (including $t=0, \infty$) with the \texttt{Pari/GP}
  command
\begin{verbatim}
? MagmaInput([7, 11, 13, 17, 19, 29, 31, 41, 59, 61, 71, 79, 89, 101, 109, 131,
139, 149, 151, 179, 181, 191, 199, 211, 229, 239, 241, 251, 269, 271, 281, 311,
331, 349, 359, 379, 389])
\end{verbatim}
  This creates a file called ``\texttt{Data.txt}'' (available in \cite{github}) needed to run the
  elimination script in \texttt{Magma}. In \texttt{Magma}, we load the file and run the
  following command.
  {\fontsize{11}{13}\selectfont
  	
\begin{verbatim}
> load "Data.txt";
> time TheoremA(2,2,Data);
The newforms [ 3, 9, 12 ] cannot be discarded using primes in [ 7, 11, 13, 17, 
19, 29, 31, 41, 59, 61, 71, 79, 89, 101, 109, 131, 139, 149, 151, 179, 181, 191,
199, 211, 229, 239, 241, 251, 269, 271, 281, 311, 331, 349, 359, 379, 389 ].
The rest of the newforms can be discarded for p outside  [ 2, 3, 5, 11, 19, 29 ]
Time: 57.710
\end{verbatim}
  } This proves that except for the forms indexed by $3, 9, 12$, all
  other newforms in the space $S_2(\Gamma_0(3^2\cdot (\sqrt{5})^2))$
  can be discarded if $p \not \in \{2, 3, 5, 11, 19, 29\}$. We are led
  to prove that the three forms that cannot be eliminated in fact have
  complex multiplication. It is not hard to verify that the first form
  has complex multiplication by $\Q(\sqrt{-3})$ and the other two ones
  by $\Q(\zeta_{15})/K$. The way to verify this is simple: take the
  first form $\newform$, and compute its quadratic twist by
  $\sqrt{-3}$. If it equals $\newform$ then it has complex
  multiplication by $\sqrt{-3}$. Otherwise, it should correspond to
  either another newform in the same space, or a newform in a smaller
  space. The space of Hilbert modular forms of level $(\sqrt{5})^2$ is
  the zero space, while the space of level $3 \cdot (\sqrt{5})^2$ has
  three newforms. By looking at the Fourier coefficient at a prime
  ideal $\idK$ dividing $11$ it is easy to verify that none of them
  corresponds to a twist of $\newform$ (because
  $a_{\idK}(\newform)=0$, while the others are nonzero). Similarly,
  looking at the Fourier coefficients of newforms in
  $S_2(\Gamma_0(3^2\cdot(\sqrt{5})^2))$, it is easy to verify that if
  $\idK$ is a prime ideal dividing $79$ then $a_{\idK}(\newform)=16$
  while no other form has eigenvalue $\pm 16$ at $\idK$. A similar
  computation proves the result for the other two forms.

  The output for the other three spaces (including the running time)
  is available at \cite{github}. We summarize the needed information:
  \begin{enumerate}
  \item All newforms in $S_2(\Gamma_0(3^2\cdot (\sqrt{5})^3))$ except the ones labeled $\{ 1, 7, 11, 12, 13, 16, 21\}$ can be discarded when $p \not \in \{2, 3, 5, 7, 11, 13, 19, 29, 31\}$.
    
  \item All newforms in $S_2(\Gamma_0(3^3\cdot (\sqrt{5})^2))$ except the ones labeled $\{  64, 65, 69, 73, 77, 78, 79\}$ can be discarded when $p \not \in \{2,3,5,7,11,19,29,41,61 \}$.
    
  \item All newforms in $S_2(\Gamma_0(3^3\cdot (\sqrt{5})^3))$ except the ones labeled $\{  22, 39\}$ can be discarded when $p \not \in \{2, 3, 5, 7, 11, 13, 19, 29, 31, 41, 61, 71, 79, 89, 101, 109\}$.
  \end{enumerate}
%
%
%
%
%
%
  The seven newforms that cannot be eliminated in (1) have complex
  multiplication by $\Q(\sqrt{-3})$. One of them corresponds to the
  trivial solution $(\pm 1,\mp 1,0)$, as proved in
  Lemma~\ref{lemma:cond-at-t=inf} (see also Remark \ref{rem:cond-trivial}), and another is its quadratic twist by
  $\Q(\zeta_{15})$. 
  The seven forms in (2) have complex multiplication
  by $\Q(\zeta_{15})$. One of them arises from the trivial solution
  $(0, \pm 1, \mp 1)$, by Lemma~\ref{lemma:cond-at-t=0} and Remark \ref{rem:cond-trivial},  and another is its quadratic twist by $\sqrt{-3}$. 
  Finally,  the two newforms in (3) that cannot be discarded correspond to the
  stated ghost solutions.
\end{proof}

\begin{remark}\label{remar:Kraus}
  While studying other families of exponents (like the family
  $(p,p,r)$) once a finite set of primes $\bfP$ is obtained, one might
  try to use an idea due to Kraus (as explained in \cite{Kraus33p}) to
  prove non-existence of solutions for primes in $\bfP$. Kraus' idea
  consists on given a particular value of $p$, chose a prime $\ell$
  satisfying that $p \mid \ell-1$. Then one expects that not all
  values $\tilde{t}_0 \in \F_\ell \setminus\{0,1\}$ in Case (1) come
  from elements of $S_\ell$, so computing the particular set $S_\ell$
  gives an improvement on the elimination procedure. Take for example
  the generalized Fermat equation with signature $(7,7,5)$, and take $\ell=71$. Then the
  only values for $t_0$ in Case (1) are
  $\{4, 12, 19, 31, 36, 41, 53, 60, 68\}$ (there are only 9 possible
  values instead of 69).

  It looks like Kraus' idea does not give any improvement while
  working with three different prime exponents. We computed for many pairs
  $(p,\ell)$, with $15p \mid \ell-1$, the set $S_\ell$ and the
  corresponding value $\tilde{t}_0$ for the exponents $(5,p,3)$. In
  all cases the resulting set was the whole
  $\F_{\ell}\setminus\{0,1\}$. 
\end{remark}

The second stated result follows readily from Theorem \ref{thmA}.

\begin{thmB}
Let $\bfP=\{2, 7, 11, 13, 19, 29, 31, 41, 61\}$. Then if $p$ is a prime
  number not in $\bfP$, there are no primitive solutions $(a,b,c)$ to
  \[
 x^5 + y^p + z^3 =0
 \]
 with $b$ odd and $3\mid b$ or $5 \mid b$.
\end{thmB}

\begin{proof}
  Let $(a,b,c)$ be a solution to the equation satisfying that
  $3 \mid b$ or $5 \mid b$ and $p>13$. Since $b$ is odd, the
  representation $\bar{\rho}_{\id{p}}^+$ is absolutely irreducible (by
  Corollary~\ref{coro:large-image-5p3} since $C(2)=13$). By
  Proposition~\ref{prop:Igusa} (and Remark~\ref{rem:pot-good-red}) the
  local type at $3$ or at $(\sqrt{5})$ (depending on whether
  $3 \mid b$ or $5 \mid b$) is special. Then we can lower the level at
  all primes of $K$ dividing $abc$, keeping the same local type as
  $\rho_{\id{p}}^+$ at primes dividing $3\cdot(\sqrt{5})$. The local
  type being special implies that there exists a Hilbert newform
  $\newform$ without complex multiplication in
  $S_2(\Gamma_0(3^{\varepsilon_3}\cdot(\sqrt{5})^{\varepsilon_5}))$
  congruent to $\rho_{\id{p}}^+$, with $\varepsilon_3=2$ or
  $\varepsilon_5=2$. This contradicts Theorem \ref{thmA}, since the fact that
  $\newform$ does not have complex multiplication rules out the second
  case, and the fact that one exponent is $2$ rules out the third
  one. The set $\bfP$ is the one obtained while discarding forms in
  the spaces with $\varepsilon_3=2$ or $\varepsilon_5=2$. Note that we can also remove primes $3$ and $5$ from $\bfP$, since those cases were solved in \cite{Bruin,Poonen}.
\end{proof}

\subsection{On the Local type at $3$}
Let $p>3$ be a prime number and let $\id{p}$ be a prime ideal of $K$
dividing $p$.  The goal of this section is to compute the local type
(see Definition \ref{def:local-type}) at $3$ of the Galois
representation $\rho_{\id{p}}:=\rho_{\id{p}}^+$ (for different specializations of the
parameter) in order to rule out the third case in Theorem \ref{thmA} by
imposing a local condition on the solution. Let $\Q_9$ denote the
quadratic unramified extension of $\Q_3$.

\vspace{3pt}

Let $\rho_i : \Gal_K \to \GL_2(K_{\id{p}})$, $i=1,2$, denote the
$\id{p}$-adic Galois representations attached to
$\mot_1:=\HGM^+(-1/8)$ and $\mot_2:=\HGM^+(9/8)$ respectively.

\begin{lemma}
  \label{lemma:type-ghotsol}
  The local type of $\rho_i$ (for $i=1,2$) at the prime ideal $3$ is
  supercuspidal, given by the induction of a character of conductor 3
  from a quadratic ramified extension $L$ of $\Q_9$.
\end{lemma}
\begin{proof}
  Since the conductor exponent at $3$ of $\mot_i$ (for $i=1,2$) is $3$
  and its Nebentypus is trivial, its local type is supercuspidal,
  obtained as the induction of a character of conductor $3$ from a
  quadratic ramified extension $L$ of $\Q_9$ (see for example
  \cite[Corollary 3.1]{MR3056552}).
\end{proof}

\begin{lemma}
  \label{lemma:same-type-3}
  If $\bar{\rho}_\id{p} \simeq \bar{\rho}_i$, for $i=1,2$, then the
  local type of ${\rho}_\id{p}$ at $3$ is also supercuspidal, induced
  from the same quadratic extension $L$.
\end{lemma}
%
%
%
\begin{proof}	
  Follows from Lemma \ref{lemma:type-ghotsol} and \cite[Proposition
  1.1]{MR4583916}.
\end{proof}

The following theorem gives sufficient conditions to prove $\bar{\rho}_\id{p}\not\simeq \bar{\rho}_i$, for $i=1,2$.
\begin{thm}
  \label{thm:local-type-3}
  Let $(a,b,c)$ be a primitive solution to~(\ref{eq:5p3}) with $p>3$ satisfying
  that $3 \nmid c$.  Set $t_0=-\frac{a^5}{c^3}$. Then $\HGM^+(t_0)$
  is not congruent modulo $p$ to $\HGM^+(-1/8)$ nor
  to $\HGM^+(9/8)$.
\end{thm}

\begin{proof}

  Let $K$ be a finite extension of $\Q_p$.
  Let $\bfC$ be a hyperelliptic curve defined over $K$ by an equation
\[
 \bfC: y^2=F(x).
\]
Then by \cite[Theorem 1.9]{MR4566695}, if $\Jac(\bfC)$ has potentially
good reduction, it attains good reduction over an at most quadratic
ramified extension of $K({\mathcal R})$, the splitting field of $F(x)$
(where ${\mathcal R}$ is the set of roots of $F(x)$, following the
authors' notation).  In our case, from (\ref{eq:hyperell-5-p-3}) we have $F(x)=5x^6 - 12x^5 + 10tx^3 + t^2$ and $K=\Q_9$.

For $t_0=-\frac{1}{8}$ we obtain the polynomial
$5x^6-12x^5 -\frac{5}{4}x^3 + \frac{1}{64}$, giving the same extension
as the \emph{reduced} polynomial $x^6+ 3x^4+ 9$ (a reduction algorithm
is implemented in \cite{MR2194887}). The
isomorphisms of fields can also be verified with the \texttt{Pari/GP}
implementation of Panayi's algorithm available in \cite{MR2194887}.	
Its Galois closure $\Q_9({\mathcal R})$ has Galois group $D_{12}$, the
dihedral group with $12$ elements, and ramification degree $6$. To
verify whether the Jacobian of $\bfC_{5,3}^+(-1/8)$ attains good
reduction over $\Q_9({\mathcal R})$ or a quadratic extension of it, we
can compute its cluster picture in \texttt{Magma} using Tim Dokchitser's implementation.
\begin{verbatim}
> ClusterPicture(5*x^6 - 12*x^5 - 5/4*x^3 + 1/64,3);     
((1,2,3),(4,5,6)) F[[7,8]] d=[5/6,5/6,0]
\end{verbatim}
Note that over $\Q_9({\mathcal R})$ it is true that all distances are
integral (since $e=6$), but the last condition of the semistability
criterion (\cite[Definition 1.7]{MR4566695}) is not
satisfied. 
By Lemma \ref{lemma:type-ghotsol}, the local inertial type of
$\Jac(\bfC_{5,3}^+(-1/8))$ at $3$ is the induction of an order $6$
character $\varkappa$ from one of the quadratic ramified extensions of
$\Q_9$.  There are two such extensions, namely $L_1=\Q_9(\sqrt{3})$
and $L_2=\Q_9(\sqrt{3\alpha})$, where $\alpha^2=-1$ in $\Q_9$.

Then the field $\Q_9({\mathcal R})$ is the fixed field of the local
Galois representation obtained from the induction of the character
$\varkappa^2$. We can compute the degree $12$ polynomial corresponding to
$\Q_9({\mathcal R})/\Q_3$ in three different ways:
\begin{itemize}
\item by a direct computation (compute the compositum of the
  polynomial with itself and compute its $3$-adic factors),
\item search the database of David Roberts and John Jones
  \cite{MR2194887} (see also \cite{2507.02360}) for degree $12$
  extensions and eliminate all of them but one,
\item use the fact that our inertial type satisfies the properties
  listed in the introduction of \cite{Nuno}, so the representation
  appears as the $4$-torsion of an elliptic curve (see Theorem 1.1 of
  loc. cit) that is computed by the authors.
\end{itemize}
Using any of these three approaches, the Galois closure is obtained by
adding to $\Q_3$ a root of the degree $12$ polynomial
\[
  P=x^{12}+3x^{11}+3x^{10}-6x^9+3x^8+9x^7+9x^4+9x^3+9.
\]
Using the polynomial $P$ it is easy to verify that
$L_1 \subset \Q_9({\mathcal R})$. We mimic the same computation for
$t_0=\frac{9}{8}$ and get that the reduced polynomial also equals
$x^6+ 3x^4+ 9$. Then the two ghost solutions share the same degree $12$ extension.
%
%
%
%
%
%
%
%
%

Let $(a,b,c)$ be a solution to~(\ref{eq:5p3}) with $3 \nmid c$.
According to Table~\ref{table:cond3} it must be the case that either
$t_0 \equiv -1 \pmod 9$ or $t_0 \equiv 2 \pmod 9$. In both cases, the
reduced polynomial equals $x^6+ 3x^4+ 6x^3+ 9x^2+ 63x+ 9$, and its
Galois closure corresponds to the field extension obtained by adding a
root of
$Q=x^{12}+6x^{11}+21x^{10}+36x^9+30x^8+36x^7+3x^6+36x^5+27x^4-9x^2+36$. But
the extension obtained from $P$ is not isomorphic to that obtained
from $Q$.
\end{proof}

\begin{remark}
  When $3 \mid c$, it occurs sometimes that the local type of $\rho_{\id{p}}$ matches that of the ghost solutions.
\end{remark}

\begin{thmC}
Let $\bfP$ be as in Theorem \ref{thmA}. Assume that
  Conjecture~\ref{conj:large-image} holds for $L=K=\Q(\sqrt{5})$ with
  constant $C$. Then if $p$ is a prime number satisfying that
  $p \not \in \bfP$ and $p > C$, there
  are no non-trivial primitive solutions $(a,b,c)$ to
 \[
 x^5 + y^p + z^3 =0
 \]
 with $3\nmid c$.
\end{thmC}

\begin{proof}
  Let $(a,b,c)$ be a putative non-trivial solution to~(\ref{eq:5p3})
  with $3\nmid c$. Set $t_0=-a^5/c^3$.   By
  Corollary~\ref{coro:Curve-with-non-CM}, the Jacobian of $\bfC_{5,3}^+(t_0)$ has a prime of potentially multiplicative reduction.
  
   If the Jacobian of $\bfC_{5,3}^+(t_0)$ is isogenous over $\CC$ to the product of two elliptic curves, then either the curves do not have complex multiplication (in which case Conjecture \ref{conj:large-image} is proven) or they do have complex multiplication and so all bad primes have potentially good reduction, leading to a contradiction.  Theorefore, we can assume that the Jacobian of $\bfC_{5,3}^+(t_0)$ is absolutely simple and does not have complex nor quaternionic multiplication. Then
  Conjecture~\ref{conj:large-image} implies that if $p > C$,
  $\bar{\rho}_{\id{p}}$ cannot be reducible nor lie in the normalizer
  of a Cartan group. Theorem \ref{thmA} then implies that $\bar{\rho}_{\id{p}}$ is
  isomorphic to $\bar{\rho}_1$ or $\bar{\rho}_2$, which contradicts
  Theorem~\ref{thm:local-type-3}.
\end{proof}

\bibliographystyle{alpha}
	\bibliography{biblio1}

\end{document}